\documentclass{svjour3}                     
\smartqed  
\usepackage{graphicx}
\usepackage{amsfonts}
\usepackage{amssymb}
\usepackage{amsmath}
\usepackage{url}
\usepackage{color}
\usepackage[dvipsnames]{xcolor}
\usepackage{epstopdf}
\usepackage{bm}
\usepackage{cases}
\usepackage{paralist}
\usepackage{multirow}
\usepackage{array}
\usepackage{enumitem}
\usepackage{mathtools}
\usepackage{empheq}
 
\usepackage{subcaption}
\usepackage{natbib}
 \bibpunct[, ]{(}{)}{,}{a}{}{,}

 \def\bibhang{24pt}

\newcommand{\argmin}{\mathop{\text{arg\,min}}}

\newcommand{\st}{\mathop{\text{s.t.}}}

\newcolumntype{C}{>{\centering\arraybackslash}p{4em}}

\newcounter{observation}

\begin{document}

\title{Scenario Reduction Revisited: \\ Fundamental Limits and Guarantees}



\author{Napat~Rujeerapaiboon, Kilian~Schindler, Daniel~Kuhn, Wolfram~Wiesemann}


\institute{Napat~Rujeerapaiboon, Kilian~Schindler, Daniel~Kuhn\at
			  Risk Analytics and Optimization Chair\\
              \'Ecole Polytechnique F\'ed\'erale de Lausanne, Switzerland \\
              Tel.: +41 (0)21 693 00 36 \
              Fax: +41 (0)21 693 24 89\\
              \email{napat.rujeerapaiboon@epfl.ch, kilian.schindler@epfl.ch, daniel.kuhn@epfl.ch}
           \and
           Wolfram~Wiesemann\at
           Imperial College Business School\\
           Imperial College London, United Kingdom\\
           Tel.: +44 (0)20 7594 9150\\
           \email{ww@imperial.ac.uk}
}


\maketitle

\begin{abstract}
The goal of scenario reduction is to approximate a given discrete distribution with another discrete distribution that has fewer atoms. We distinguish continuous scenario reduction, where the new atoms may be chosen freely, and discrete scenario reduction, where the new atoms must be chosen from among the existing ones. Using the Wasserstein distance as measure of proximity between distributions, we identify those $n$-point distributions on the unit ball that are least susceptible to scenario reduction, {\em i.e.}, that have maximum Wasserstein distance to their closest $m$-point distributions for some prescribed $m<n$. We also provide sharp bounds on the added benefit of continuous over discrete scenario reduction. Finally, to our best knowledge, we propose the first polynomial-time constant-factor approximations for both discrete and continuous scenario reduction as well as the first exact exponential-time algorithms for continuous scenario reduction.



\end{abstract}


\section{Introduction}
\label{section:intro}

The vast majority of numerical solution schemes in stochastic programming rely on a discrete approximation of the true (typically continuous) probability distribution governing the uncertain problem parameters. This discrete approximation is often generated by sampling from the true distribution. Alternatively, it could be constructed directly from real historical observations of the uncertain parameters. To obtain a faithful approximation for the true distribution, however, the discrete distribution must have a large number $n$ of support points or {\em scenarios}, which may render the underlying stochastic program computationally excruciating. 

An effective means to ease the computational burden is to rely on \textit{scenario reduction} pioneered by \cite{Dupacova2003}, which aims to approximate the initial $n$-point distribution with a simpler $m$-point distribution ($m< n$) that is as close as possible to the initial distribution with respect to a probability metric; see also \cite{Heitsch2003}. The modern stability theory of stochastic programming surveyed by \cite{Dupacova1990} and \cite{Roemisch2003} indicates that the Wasserstein distance may serve as a natural candidate for this probability metric.

Our interest in Wasserstein distance-based scenario reduction is also fuelled by recent progress in data-driven distributionally robust optimization, where it has been shown that the worst-case expectation of an uncertain cost over all distributions in a Wasserstein ball can often be computed efficiently via convex optimization \citep{MohajerinEsfahani2015, Zhao2015, Gao2016}. A Wasserstein ball is defined as the family of all distributions that are within a certain Wasserstein distance from a discrete reference distribution. As distributionally robust optimization problems over Wasserstein balls are harder to solve than their stochastic counterparts, we expect significant computational savings from replacing the initial $n$-point reference distribution with a new $m$-point reference distribution. The benefits of scenario reduction may be particularly striking for two-stage distributionally robust linear programs, which admit tight approximations as semidefinite programs \citep{Hanasusanto2016}.


Suppose now that the initial distribution is given by $\mathbb{P} = \sum_{i\in I} p_i \delta_{\bm\xi_i}$, where $\bm \xi_i\in\mathbb R^d$ and $p_i\in [0,1]$ represent the location and probability of the $i$-th scenario of $\mathbb{P}$ for $i\in I = \{ 1, \hdots, n \}$. Similarly, assume that the reduced target distribution is representable as $\mathbb{Q} = \sum_{j\in J} q_j \delta_{\bm\zeta_j}$, where $\bm \zeta_j\in\mathbb R^d$ and $q_j\in[0,1]$ stand for the location and probability of the $j$-th scenario of $\mathbb Q$ for $j\in J=\{1, \hdots, m \}$. 
Then, the type-$l$ Wasserstein distance between $\mathbb{P}$ and $\mathbb Q$ is defined through
\begin{equation*}
\begin{aligned}
	d_l (\mathbb{P}, \mathbb{Q}) = \left[ \min_{\mathbf{\Pi} \in \mathbb{R}_+^{n \times m}} \left\{ \sum_{i\in I} \sum_{j\in J} \pi_{ij} \Vert \bm\xi_i - \bm\zeta_j \Vert^l:
	\begin{array}{l}
		\sum_{j\in J} \pi_{ij} = p_i \ \forall i \in I  \\[2mm]
		\sum_{i\in I} \pi_{ij} = q_j \ \forall j \in J
	\end{array}
	\right\} \right]^{1/l},
\end{aligned}
\end{equation*}
where $l\geq 1$ and $\|\cdot\|$ denotes some norm on $\mathbb{R}^d$, see, {\em e.g.}, \cite{Heitsch2007} or \cite{Pflug2011}. The linear program in the definition of the Wasserstein distance can be viewed as a minimum-cost transportation problem, where $\pi_{ij}$ represents the amount of probability mass shipped from~$\bm\xi_i$ to~$\bm\zeta_j$ at unit transportation cost $\| \bm\xi_i - \bm\zeta_j \|^l$. Thus, $d_l^l (\mathbb{P}, \mathbb{Q})$ quantifies the minimum cost of moving the initial distribution $\mathbb{P}$ to the target distribution $\mathbb{Q}$. 

For any $\mathrm{\Xi}\subseteq \mathbb{R}^d$, we denote by $\mathcal{P}_\mathrm{E}(\mathrm{\Xi},n)$ the set of all uniform discrete distributions on $\mathrm{\Xi}$ with exactly $n$ distinct scenarios and by $\mathcal{P}(\mathrm{\Xi}, m)$ the set of all (not necessarily uniform) discrete distributions on $\mathrm{\Xi}$ with at most $m$ scenarios. We henceforth assume that $\mathbb{P}\in \mathcal{P}_\mathrm{E}(\mathbb{R}^d,n)$. 
This assumption is crucial for the simplicity of the results in Sections~\ref{section:wcdist} and~\ref{section:compare}, and it is almost surely satisfied whenever $\mathbb{P}$ is obtained via sampling from a continuous probability distribution. Hence, we can think of $\mathbb{P}$ as an {\em empirical distribution}. To remind us of this interpretation, we will henceforth denote the initial distribution by $\hat{\mathbb{P}}_n$. Note that the pairwise difference of the scenarios can always be enforced by slightly perturbing their locations, while the uniformity of their probabilities can be enforced by decomposing the scenarios into clusters of close but mutually distinct sub-scenarios with (smaller) uniform probabilities.

We are now ready to introduce the \textit{continuous scenario reduction problem}
\begin{equation*}
	C_l(\hat{\mathbb{P}}_n, m) \ = \ \min_{\mathbb{Q}} \ \left\{ d_l(\hat{\mathbb{P}}_n, \mathbb{Q}): \, \mathbb{Q} \in \mathcal{P}(\mathbb{R}^d,m) \right\},
\end{equation*}
where the new scenarios $\bm\zeta_j$, $j \in J$, of the target distribution $\mathbb{Q}$ may be chosen freely from within $\mathbb{R}^d$, as well as the \emph{discrete scenario reduction problem}
\begin{equation*}
	D_l(\hat{\mathbb{P}}_n, m) \ = \ \min_{\mathbb{Q}} \ \left\{ d_l(\hat{\mathbb{P}}_n, \mathbb{Q}): \, \mathbb{Q} \in \mathcal{P}(\text{supp}(\hat{\mathbb{P}}_n),m) \right\},
\end{equation*}
where the new scenarios must be chosen from within the support of the empirical distribution, which is given by the finite set $\text{supp}(\hat{\mathbb{P}}_n)=\{ \bm\xi_i: i \in I \}$. Even though the continuous scenario reduction problem offers more flexibility and is therefore guaranteed to find (weakly) better approximations to the initial empirical distribution, to our best knowledge, the existing stochastic programming literature has exclusively focused on the discrete scenario reduction problem. 

Note that if the support points $\bm\zeta_j$, $j \in J$, are fixed, then both scenario reduction problems simplify to a linear program over the probabilities $q_j$, $j\in J$, which admits an explicit solution \cite[Theorem~2]{Dupacova2003}. Otherwise, however, both problems are intractable. Indeed, if $l = 1$, then the discrete scenario reduction problem represents a metric $k$-median problem with $k = m$, which was shown to be $\mathcal{NP}$-hard by \citet{Kariv1979}. If $l = 2$ and distances in $\mathbb R^d$ are measured by the 2-norm, on the other hand, then the continuous scenario reduction problem constitutes a $k$-means clustering problem with $k = m$, which is $\mathcal{NP}$-hard even if $d = 2$ or $m = 2$; see \cite{Mahajan2009} and \cite{Aloise2009}. 

\cite{Heitsch2003} have shown that the discrete scenario reduction problem admits a reformulation as a mixed-integer linear program (MILP), which can be solved to global optimality for $n\lesssim 10^3$ using off-the-shelf solvers. For larger instances, however, one must resort to approximation algorithms. Most large-scale discrete scenario reduction problems are nowadays solved with a greedy heuristic that was originally devised by \cite{Dupacova2003} and further refined by \cite{Heitsch2003}. For example, this heuristic is routinely used for scenario (tree) reduction in the context of power systems operations, see, {\em e.g.}, \cite{Romisch2010} or \cite{Morales2009} and the references therein. Despite its practical success, we will show in Section~\ref{section:algorithm} that this heuristic fails to provide a constant-factor approximation for the discrete scenario reduction problem.

This paper extends the theory of scenario reduction along several dimensions.
\begin{itemize}
	\item[(i)] We establish fundamental performance guarantees for continuous scenario reduction when $l\in\{1,2\}$, 
	{\em i.e.}, we show that the Wasserstein distance of the initial $n$-point distribution to its nearest $m$-point distribution is bounded by $\sqrt{\frac{n-m}{n-1}}$ across all initial distributions on the unit ball in $\mathbb R^d$. We show that for $l=2$ this worst-case performance is attained by some initial distribution, which we construct explicitly. We also provide evidence indicating that this worst-case performance reflects the norm rather than the exception in high dimensions $d$. Finally, we provide a lower bound on the worst-case performance for $l=1$.
    \item[(ii)]
    We analyze the loss of optimality incurred by solving the discrete scenario reduction problem instead of its continuous counterpart. Specifically, we demonstrate that the ratio $D_l(\hat{\mathbb{P}}_n,m)/C_l(\hat{\mathbb{P}}_n,m)$ is bounded by $\sqrt{2}$ for $l=2$ and by~2 for $l=1$. We also show that these bounds are essentially tight.
    \item[(iii)]
    We showcase the intimate relation between scenario reduction and $k$-means clustering. By leveraging existing constant-factor approximation algorithms for $k$-median clustering problems due to \cite{Arya2004} and the new performance bounds from~(ii), we develop the first polynomial-time constant-factor approximation algorithms for both continuous and discrete scenario reduction. We also show that these algorithms can be warmstarted using the greedy heuristic by \cite{Dupacova2003} to improve practical performance.
    \item[(iv)]
    We present exact mixed-integer programming reformulations for the continuous scenario reduction problem. 
\end{itemize}



Continuous scenario reduction is intimately related to the optimal quantization of probability distributions, where one seeks an $m$-point distribution approximating a non-discrete initial distribution. Research efforts in this domain have mainly focused on the asymptotic behavior of the quantization problem as $m$ tends to infinity, see \cite{Graf2000}. The ramifications of this stream of literature for stochastic programming are discussed by \cite{Pflug2011}. Techniques familiar from scenario reduction lend themselves also for scenario generation, where one aims to construct a scenario tree with a prescribed branching structure that approximates a given stochastic process with respect to a probability metric, see, {\em e.g.}, \cite{Pflug2001} and \cite{Hochreiter2007}.

The rest of this paper unfolds as follows. Section~\ref{section:wcdist} seeks to identify $n$-point distributions on the unit ball that are least susceptible to scenario reduction, {\em i.e.}, that have maximum Wasserstein distance to their closest $m$-point distributions, and Section~\ref{section:compare} discusses sharp bounds on the added benefit of continuous over discrete scenario reduction. Section~\ref{section:algorithm} presents exact exponential-time algorithms as well as polynomial-time constant-factor approximations for scenario reduction. Section~\ref{section:experiment} reports on numerical results for a color quantization experiment. 
Unless otherwise specified, below we will always work with the 2-norm on $\mathbb R^d$. 

\paragraph{Notation:} We let $\mathbb I$ be the identity matrix, $\mathbf e$ the vector of all ones and $\mathbf e_i$ the $i$-th standard basis vector of appropriate dimensions. The $ij$-th element of a matrix $\mathbf{A}$ is denoted by $a_{ij}$. For $\mathbf{A}$ and $\mathbf{B}$ in the space $\mathbb S^n$ of symmetric $n \times n$ matrices, the relation $\mathbf{A} \succeq \mathbf{B}$ means that $\mathbf{A} - \mathbf{B}$ is positive semidefinite. Generic norms are denoted by $\Vert \cdot \Vert$, while $\Vert \cdot \Vert_p$ stands for the $p$-norm, $p\geq 1$. For $\mathrm{\Xi}\subseteq \mathbb{R}^d$, we define $\mathcal{P}(\mathrm{\Xi},m)$ as the set of all probability distributions supported on at most $m$ points in~$\mathrm{\Xi}$ and $\mathcal{P}_\mathrm{E}(\mathrm{\Xi},n)$ as the set of all \emph{uniform} distributions supported on exactly $n$ {\em distinct} points in~$\mathrm{\Xi}$. The support of a probability distribution~$\mathbb{P}$ is denoted by $\text{supp}(\mathbb{P})$, and the Dirac distribution concentrating unit mass at $\bm\xi$ is denoted by $\delta_{\bm\xi}$.

\section{Fundamental Limits of Scenario Reduction}
\label{section:wcdist}

In this section we characterize the Wasserstein distance $C_l (\hat{\mathbb{P}}_n, m)$ between an $n$-point empirical distribution $\hat{\mathbb{P}}_n = \frac{1}{n} \sum_{i=1}^n \delta_{\bm{\xi}_i}$ and its continuously reduced optimal $m$-point distribution $\mathbb{Q} \in \mathcal{P} (\mathbb{R}^d, m)$. Since the positive homogeneity of the Wasserstein distance $d_l$ implies that $C_l (\hat{\mathbb{P}}'_n, m) = \lambda \cdot C_l (\hat{\mathbb{P}}_n, m)$ for the scaled distribution $\hat{\mathbb{P}}'_n = \frac{1}{n} \sum_{i=1}^n \delta_{\lambda \bm{\xi}_i}$, $\lambda \in \mathbb{R}_+$, we restrict ourselves to empirical distributions $\hat{\mathbb{P}}_n$ whose scenarios satisfy $\left \lVert \bm{\xi}_i \right \rVert_2 \leq 1$, $i = 1, \ldots, n$. We thus want to quantify
\begin{equation}
\label{opt:wcwd}
\begin{aligned}
	\overline{C}_l(n,m) = \max_{\hat{\mathbb{P}}_n \in \mathcal{P}_\mathrm{E}(\mathbb{R}^d, n)} 
    \left\{ 	
    	C_l (\hat{\mathbb{P}}_n, m) \, : \, \Vert \bm\xi \Vert_2 \leq 1 \;\;
        \forall \bm\xi \in \text{supp}(\hat{\mathbb{P}}_n)
    \right\},
\end{aligned}
\end{equation}
which amounts to the worst-case ({\em i.e.}, largest) Wasserstein distance between any $n$-point empirical distribution $\hat{\mathbb{P}}_n$ over the unit ball and its optimally selected continuous $m$-point scenario reduction. By construction, this worst-case distance satisfies $\overline{C}_l(n,m) \geq 0$, and the lower bound is attained whenever $n = m$. One also verifies that $\overline{C}_l(n,m) \leq \overline{C}_l (n,1) \leq 1$ since the Wasserstein distance to the Dirac distribution $\delta_{\bm{0}}$ is bounded above by $1$. Our goal is to derive possibly tight upper bounds on $\overline{C}_l(n,m)$ for the Wasserstein distances of type $l \in \{ 1, 2 \}$.

In the following, we denote by $\mathfrak{P}(I,m)$ the family of all $m$-set partitions of the index set $I$, {\em i.e.},
\begin{equation*}
	 \mathfrak{P}(I,m) = \big\{ \{I_1, \hdots, I_m \} \; : \;
        \emptyset \neq I_1, \hdots, I_m \subseteq I, \;\; 
        \cup_j I_j = I, \;\; I_i \cap I_j = \emptyset \;\; \forall i \neq j \big\},
\end{equation*}
and an element of this set ({\em i.e.}~a specific $m$-set partition) as $\{ I_j \} \in \mathfrak{P}(I,m)$. Our derivations will make extensive use of the following theorem. 

\begin{theorem}
\label{thm:partition}
For any type-$l$ Wasserstein distance induced by any norm $\Vert \cdot \Vert$, the continuous scenario reduction problem can be reformulated as
\begin{equation}
\label{eq:voronoi-1}
	C_l(\hat{\mathbb{P}}_n, m) = \min_{\{ I_j \} \in \mathfrak{P}(I,m)} \, \left[ \frac{1}{n} \sum_{j \in J} \min_{\bm\zeta_j \in \mathbb{R}^d} \sum_{i \in I_j} \Vert \bm\xi_i - \bm\zeta_j \Vert^l \right]^{1/l}.
\end{equation}
\end{theorem}

Problem~\eqref{eq:voronoi-1} can be interpreted as a Voronoi partitioning problem that asks for a Voronoi decomposition of $\mathbb{R}^d$ into $m$ cells whose Voronoi centroids $\bm{\zeta}_1, \ldots, \bm{\zeta}_m$ minimize the cumulative $l$-th powers of the distances to $n$ prespecified points $\bm{\xi}_1, \ldots, \bm{\xi}_n$.

\begin{proof} \emph{of Theorem~\ref{thm:partition}} $\,$
Theorem~2 of \citet{Dupacova2003} implies that the smallest Wasserstein distance between the empirical distribution $\hat{\mathbb{P}}_n \in \mathcal{P}_\mathrm{E}(\mathbb{R}^d, n)$ and any distribution $\mathbb{Q}$ supported on a finite set $\mathrm{\Xi} \subset \mathbb{R}^d$ amounts to
    \begin{equation*} 
		\min_{\mathbb{Q} \in \mathcal{P}(\mathrm{\Xi},\infty)} d_l(\hat{\mathbb{P}}_n, \mathbb{Q}) = \left[ \frac{1}{n}\sum_{i \in I} \min_{\bm\zeta \in \mathrm{\Xi}}  \Vert \bm\xi_i - \bm\zeta \Vert^l\right]^{1/l},
	\end{equation*}
	where $\mathcal{P}(\mathrm{\Xi},\infty)$ denotes the set of all probability distributions supported on the finite set $\mathrm{\Xi}$. The continuous scenario reduction problem $C_l(\hat{\mathbb{P}}_n, m)$ selects the set $\mathrm{\Xi}^\star$ that minimizes this quantity over all sets in $\mathrm{\Xi} \subset \mathbb{R}^d$ with $| \mathrm{\Xi} | = m$ elements:
    \begin{equation}
    \label{eq:voronoi-2}
    	C_l(\hat{\mathbb{P}}_n, m) = \min_{\{ \bm\zeta_j\} \subseteq  \mathbb{R}^d} \, \left[ \frac{1}{n} \sum_{i \in I} \min_{j \in J} \Vert \bm\xi_i - \bm\zeta_j \Vert^l \right]^{1/l}.
    \end{equation}
One readily verifies that any optimal solution $\{ \bm{\zeta}_1^\star, \ldots, \bm{\zeta}_m^\star \}$ to problem~\eqref{eq:voronoi-2} corresponds to an optimal solution $\{ I_1^\star, \ldots, I_m^\star \}$ to problem~\eqref{eq:voronoi-1} with the same objective value if we identify the set $I_j^\star$ with all observations $\bm\xi_i$ that are closer to $\bm\zeta_j^\star$ than any other $\bm\zeta_{j'}^\star$ (ties may be broken arbitrarily). Likewise, any optimal solution $\{ I_1^\star, \ldots, I_m^\star \}$ to problem~\eqref{eq:voronoi-1} with inner minimizers $\{ \bm{\zeta}_1^\star, \ldots, \bm{\zeta}_m^\star \}$ translates into an optimal solution $\{ \bm{\zeta}_1^\star, \ldots, \bm{\zeta}_m^\star \}$ to problem~\eqref{eq:voronoi-2} with the same objective value.
    \qed
\end{proof}

\begin{remark} \emph{(Minimizers of~\eqref{eq:voronoi-1})}
\label{rem:center}
For $l = 2$, the inner minimum corresponding to the set $I_j$ is attained by the \emph{mean} $\bm\zeta^\star_j = \text{mean} (I_j) = \frac{1}{\vert I_j \vert} \sum_{i \in I_j} \bm{\xi}_i$. Likewise, for $l = 1$, the inner minimum corresponding to the set $I_j$ is attained by any \emph{geometric median}
\begin{equation*}
\bm\zeta^\star_j = \text{gmed}(I_j) \in \mathop{\arg \min}_{\bm\zeta_j \in \mathbb{R}^d} \sum_{i \in I_j} \Vert \bm\xi_i - \bm\zeta_j \Vert,
\end{equation*}
which can be determined efficiently by solving a second-order cone program whenever a $p$-norm with rational $p \geq 1$ is considered (\cite{Alizadeh2003}).
\end{remark}

The rest of this section derives tight upper bounds on $\overline{C}_l(n,m)$ for Wasserstein distances of type $l = 2$ (Section~\ref{section:wcdist-2}) as well as upper and lower bounds for Wasserstein distances of type $l = 1$
(Section~\ref{section:wcdist-1}). We summarize and discuss our findings in Section~\ref{subsec:discuss}.



\subsection{Fundamental Limits for the Type-2 Wasserstein Distance}
\label{section:wcdist-2}

We now derive a revised upper bound on $\overline{C}_l(n,m)$ for the type-$2$ Wasserstein distance. The result relies on auxiliary lemmas that are relegated to the appendix.

\begin{theorem}
\label{thm:wcwd-2}
The worst-case type-$2$ Wasserstein distance satisfies $\overline{C}_2(n,m) \leq \sqrt{\frac{n-m}{n-1}}$.
\end{theorem}

Note that whenever the reduced distribution satisfies $m > 1$, the bound of Theorem~\ref{thm:wcwd-2} is strictly tighter than the na\"ive bound of $1$ from the previous section.

\begin{proof} \emph{of Theorem~\ref{thm:wcwd-2}} $\,$
From Theorem~\ref{thm:partition} and Remark~\ref{rem:center} we observe that
	\begin{equation*}
		\raisebox{2.5mm}{$\overline{C}_2 (n,m) \; = \;$}
		\begin{array}{c@{\quad}l}
		\displaystyle \max_{ \{\bm\xi_i\} \, \subseteq \, \mathbb{R}^d} & \displaystyle \min_{\{ I_j \} \in \mathfrak{P}(I,m)} \left[ \frac{1}{n} \sum_{j \in J} \sum_{i \in I_j} \left\Vert \bm\xi_i - \text{mean}(I_j) \right\Vert_2^2 \right]^{1/2} \\[6mm]
		\displaystyle \text{s.t.} & \displaystyle \Vert \bm\xi_i \Vert_2 \leq 1 \quad \forall i \in I.
	\end{array}
	\end{equation*}
	Introducing the epigraphical variable $\tau$, this problem can be expressed as
	\begin{equation}
		\label{opt:wcwd-2}
		\raisebox{7.75mm}{$\overline{C}_2^2 (n,m) \; = \;$}
		\begin{array}{c@{\quad}l@{\quad}l}
		\displaystyle \max_{\tau \in \mathbb{R}, \; \{\bm\xi_i\} \, \subseteq \, \mathbb{R}^d} & \displaystyle \frac{1}{n} \tau \\[2mm]
		\displaystyle \text{s.t.} & \displaystyle \tau \leq \sum_{j \in J} \sum_{i \in I_j} \left\Vert \bm\xi_i - \text{mean}(I_j) \right\Vert_2^2 & \displaystyle \forall \{ I_j \} \in \mathfrak{P}(I,m) \\[5mm]
		& \displaystyle \bm\xi_i{}^\top \bm\xi_i \leq 1 & \forall i \in I.
	\end{array}
	\end{equation}
	For each $j \in J$ and $i \in I_j$, the squared norm in the first constraint of~\eqref{opt:wcwd-2} can be expressed in terms of the inner products between pairs of empirical observations: 
	\begin{equation*}
	\mspace{-20mu}
	\begin{aligned}
		\big\| \bm\xi_i - & \text{mean}(I_j) \big\|_2^2 \;\; = \;\;
		\frac{1}{\vert I_j \vert^2} \Big\Vert \vert I_j \vert\bm\xi_i - \sum_{k \in I_j} \bm\xi_k \Big\Vert_2^2 \\
		&= \;\; \frac{1}{\vert I_j \vert^2} \Bigg( \vert I_j \vert^2 \bm\xi_i{}^\top \bm\xi_i - 2\vert I_j \vert \sum_{k \in I_j} \bm\xi_i{}^\top\bm\xi_k + \sum_{k \in I_j} \bm\xi_k{}^\top\bm\xi_k + \sum_{\substack{k,k^\prime \in I_j \\ k \neq k^\prime}} \bm\xi_{k}{}^\top \bm\xi_{k^\prime} \Bigg).
	\end{aligned}
	\end{equation*}
	Introducing the Gram matrix 
	\begin{equation}
    \label{eq:gram}
		\mathbf{S} = [\bm\xi_1, \hdots, \bm\xi_n]^\top [\bm\xi_1, \hdots, \bm\xi_n] \in \mathbb{S}^n, \;\;
		\mathbf{S} \succeq \bm{0} \, \text{ and } \text{rank}(\mathbf{S}) \leq \min \{ n, d \}
	\end{equation}
	then allows us to simplify the first constraint in~\eqref{opt:wcwd-2} to
	\begin{equation*}
\tau \leq \sum_{j \in J} \frac{1}{\vert I_j \vert^2} \sum_{i \in I_j} \Bigg( \vert I_j \vert^2 s_{ii} - 2\vert I_j \vert \sum_{k \in I_j} s_{ik} + \sum_{k \in I_j} s_{kk} + \sum_{\substack{k,k^\prime \in I_j \\ k \neq k^\prime}} s_{k k^\prime} \Bigg).
	\end{equation*}
	Note that the second constraint in problem~\eqref{opt:wcwd-2} can now be expressed as $s_{ii} \leq 1$, and hence all constraints in~\eqref{opt:wcwd-2} are linear in the Gram matrix $\mathbf{S}$.

Our discussion implies that we obtain an upper bound on $\overline{C}_2(n,m)$ by reformulating problem~\eqref{opt:wcwd-2} as a semidefinite program in terms of the Gram matrix~$\mathbf{S}$
	\begin{equation}
    \label{opt:symmetric_sdp}
		\begin{array}{c@{\quad}l@{\quad}l}
		\displaystyle \max_{\tau \in \mathbb{R}, \; \mathbf{S} \in \mathbb{S}^n} & \displaystyle \frac{1}{n} \tau \\
		\displaystyle \text{s.t.} & \displaystyle \tau \leq \sum_{j \in J} \frac{1}{\vert I_j \vert^2} \sum_{i \in I_j} \Bigg( \vert I_j \vert^2 s_{ii} - 2\vert I_j \vert \sum_{k \in I_j} s_{ik} + \sum_{k \in I_j} s_{kk} + \sum_{\substack{k,k^\prime \in I_j \\ k \neq k^\prime}} s_{k k^\prime} \Bigg) \\[7mm]
		& \displaystyle \mspace{405mu} \forall \{ I_j \} \in \mathfrak{P}(I,m) \\[1mm]
		& \displaystyle \mathbf{S} \succeq \bm{0}, \;\; s_{ii} \leq 1 \quad \forall i \in I,
	\end{array}
	\end{equation}
where we have relaxed the rank condition in the definition of the Gram matrix~\eqref{eq:gram}. Lemma~\ref{lem:symmetry} in the appendix shows that~\eqref{opt:symmetric_sdp} has an optimal solution $(\tau^\star, \mathbf{S}^\star)$ that satisfies $\mathbf{S}^\star = \alpha \mathbb{I} + \beta \mathbf{e}\mathbf{e}^\top$ for some $\alpha, \beta \in \mathbb{R}$. Moreover, Lemma~\ref{lem:eigen} in the appendix shows that any matrix of the form $\mathbf{S} = \alpha \mathbb{I} + \beta \bm{1}\bm{1}^\top$ is positive semidefinite if and only if $\alpha \geq 0$ and $\alpha + n \beta \geq 0$. We thus conclude that~\eqref{opt:symmetric_sdp} can be reformulated as
	\begin{equation}
	\label{opt:almost_there}
		\begin{array}{c@{\quad}l@{\quad}l}
		\displaystyle \max_{\tau, \alpha, \beta \in \mathbb{R}} & \displaystyle \frac{1}{n} \tau \\
		\displaystyle \text{s.t.} & \displaystyle \tau \leq (n - m) \alpha, \;\; \alpha + \beta \leq 1 \\[1mm]
		& \displaystyle \alpha \geq 0, \;\; \alpha + n \beta \geq 0,
	\end{array}
	\end{equation}
where the first constraint follows from the fact that for any set $I_j$ in~\eqref{opt:symmetric_sdp}, we have
	\begin{equation*}
		\frac{1}{\vert I_j \vert^2} \sum_{i \in I_j} \Bigg( \vert I_j \vert^2 (\alpha + \beta) - 2\vert I_j \vert \left( \alpha + \vert I_j \vert \beta \right) + \sum_{k \in I_j} (\alpha + \beta) + \sum_{\substack{k,k^\prime \in I_j \\ k \neq k^\prime}} \beta \Bigg) = (\vert I_j \vert - 1) \alpha,
	\end{equation*}
	and $\sum_{j \in J} (\vert I_j \vert - 1) \alpha = (n-m) \alpha$ since $\vert I \vert = n$ and $\vert J \vert = m$. The statement of the theorem now follows since problem~\eqref{opt:almost_there} is optimized by $\tau^\star = \frac{n(n-m)}{(n-1)}$, $\alpha^\star = \frac{n}{n-1}$ and $\beta^\star = \frac{-1}{n-1}$.
\qed
\end{proof}

The proof of Theorem~\ref{thm:wcwd-2} shows that the upper bound $\sqrt{\frac{n-m}{n-1}}$ on the worst-case type-$2$ Wasserstein distance $\overline{C}_2(n,m)$ is tight whenever there is an empirical distribution $\hat{\mathbb{P}}_n \in \mathcal{P}_\mathrm{E} (\mathbb{R}^d, n)$ whose scenarios $\bm{\xi}_1, \ldots, \bm{\xi}_n$ correspond to a Gram matrix $\mathbf{S} = [\bm\xi_1, \hdots, \bm\xi_n]^\top [\bm\xi_1, \hdots, \bm\xi_n] = \frac{n}{n-1} \mathbb{I} - \frac{1}{n-1} \mathbf{e}\mathbf{e}^\top$, which implies $\Vert \bm\xi_i \Vert_2 = \sqrt{s_{ii}} = 1$ for all $i \in I$. We now show that such an empirical distribution exists when $d \geq n - 1$.


\begin{proposition}
\label{prop:wcwd-2}
For $d \geq n - 1$, there is $\hat{\mathbb{P}}_n \in \mathcal{P}_\mathrm{E} (\mathbb{R}^d, n)$ with $\left \lVert \bm{\xi} \right \rVert_2 \leq 1$ for all $\bm{\xi} \in \text{\emph{supp}}(\hat{\mathbb{P}}_n)$ such that $C_2 (\hat{\mathbb{P}}_n, m) = \sqrt{\frac{n-m}{n-1}}$.
\end{proposition}

\begin{proof} 
Assume first that $d = n$ and consider the empirical distribution $\hat{\mathbb{P}}_n = \frac{1}{n} \sum_{i=1}^n \delta_{\bm{\xi}_i}$ defined through
    \begin{equation}
    \label{eq:wcwd_in_Rn}
        \bm\xi_i = y\mathbf{e} + (x-y)\mathbf{e}_i \in \mathbb{R}^n \quad \text{with} \quad
        x = \sqrt{\frac{n-1}{n}} \text{ and }
        y = \frac{-1}{\sqrt{n(n-1)}}.
    \end{equation}
A direct calculation reveals that $\mathbf{S} = [\bm\xi_1, \hdots, \bm\xi_n]^\top [\bm\xi_1, \hdots, \bm\xi_n] = \frac{n}{n-1} \mathbb{I} - \frac{1}{n-1} \mathbf{e}\mathbf{e}^\top$. 

To prove the statement for $d = n - 1$, we note that the $n$ scenarios in~\eqref{eq:wcwd_in_Rn} lie on the $(n-1)$-dimensional subspace $\mathcal{H}$ orthogonal to $\mathbf e\in \mathbb{R}^n$. Thus, there exists a rotation that maps $\mathcal{H}$ to $\mathbb{R}^{n-1}\times\{0\}$. As the Gram matrix is invariant under rotations, the rotated scenarios give rise to an empirical distribution $\hat{\mathbb{P}}_n \in \mathcal{P}_\mathrm{E} (\mathbb{R}^{n-1}, n)$ satisfying the statement of the proposition. Likewise, for $d > n$ the linear transformation $\bm{\xi}_i \mapsto (\mathbb{I}, \bm{0})^\top \bm{\xi}_i$, $\mathbb{I} \in \mathbb{R}^{n \times n}$ and $\bm{0} \in \mathbb{R}^{n \times (d - n)}$, generates an empirical distribution $\hat{\mathbb{P}}_n \in \mathcal{P}_\mathrm{E} (\mathbb{R}^{d}, n)$ that satisfies the statement of the proposition.
\qed
\end{proof}

Proposition~\ref{prop:wcwd-2} requires that $d \geq n - 1$, which appears to be restrictive. We note, however, that this condition is only sufficient (and not necessary) to guarantee the tightness of the bound from Theorem~\ref{thm:wcwd-2}. Moreover, we will observe in Section~\ref{subsec:discuss} that the bound of Theorem~\ref{thm:wcwd-2} provides surprisingly accurate guidance for the Wasserstein distance between practice-relevant empirical distributions $\hat{\mathbb{P}}_n$ and their continuously reduced optimal distributions.



\subsection{Fundamental Limits for the Type-1 Wasserstein Distance}
\label{section:wcdist-1}

In analogy to the previous section, we now derive a revised upper bound on $\overline{C}_l(n,m)$ for the type-$1$ Wasserstein distance. 

\begin{theorem}
\label{thm:wcwd-1}
The worst-case type-$1$ Wasserstein distance satisfies $\overline{C}_1 (n,m) \leq \sqrt{\frac{n-m}{n-1}}$.
\end{theorem}

Note that this bound is identical to the bound of Theorem~\ref{thm:wcwd-2} for $l = 2$.

\begin{proof} \emph{of Theorem~\ref{thm:wcwd-1}} $\,$
Leveraging again Theorem~\ref{thm:partition} and Remark~\ref{rem:center}, we obtain that
	\begin{equation*}
		\raisebox{4mm}{$\overline{C}_1 (n,m) \; = \;$}
		\begin{array}{c@{\quad}l}
		\displaystyle \max_{ \{\bm\xi_i\} \, \subseteq \, \mathbb{R}^d} & \displaystyle \min_{\{ I_j \} \in \mathfrak{P}(I,m)} \frac{1}{n} \sum_{j \in J} \sum_{i \in I_j} \left\Vert \bm\xi_i - \text{gmed}(I_j) \right\Vert_2 \\[5mm]
		\displaystyle \text{s.t.} & \displaystyle \Vert \bm\xi_i \Vert_2 \leq 1 \quad \forall i \in I.
	\end{array}
	\end{equation*}
We show that $\overline{C}_1 (n,m) \leq \overline{C}_2 (n,m)$ for all $n$ and $m = 1, \ldots, n$, which in turn proves the statement of the theorem by virtue of Theorem~\ref{thm:wcwd-2}. To this end, we observe that
\begin{align}
		\raisebox{4mm}{$\overline{C}_1 (n,m)$} \;\; & \raisebox{4mm}{$\leq$} \;\;
		\begin{array}{c@{\quad}l}
		\displaystyle \max_{\{\bm\xi_i\} \, \subseteq \, \mathbb{R}^d} & \displaystyle \min_{\{ I_j \} \in \mathfrak{P}(I,m)} \frac{1}{n} \sum_{j \in J} \sum_{i \in I_j} \left\Vert \bm\xi_i - \text{mean}(I_j) \right\Vert_2 \\[5mm]
		\displaystyle \text{s.t.} & \displaystyle \Vert \bm\xi_i \Vert_2 \leq 1 \quad \forall i \in I
	\end{array} \label{ineq:mean-med} \\
	& \raisebox{2.75mm}{$\leq$} \;\;
		\begin{array}{c@{\quad}l}
		\displaystyle \max_{\{\bm\xi_i\} \, \subseteq \, \mathbb{R}^d} & \displaystyle \min_{\{ I_j \} \in \mathfrak{P}(I,m)} \left[ \frac{1}{n} \sum_{j \in J} \sum_{i \in I_j} \left\Vert \bm\xi_i - \text{mean}(I_j) \right\Vert_2^2 \right]^{1/2} \\[6mm]
		\displaystyle \text{s.t.} & \displaystyle \Vert \bm\xi_i \Vert_2 \leq 1 \quad \forall i \in I,
	\end{array}
	\label{ineq:am-qm}
\end{align}
where the first inequality follows from the definition of the geometric median, which ensures that
	\begin{equation*}
		\sum_{i \in I_j} \left\Vert \bm\xi_i - \text{gmed}(I_j) \right\Vert_2 \leq \sum_{i \in I_j} \left\Vert \bm\xi_i - \text{mean}(I_j) \right\Vert_2 \quad \forall j \in J,
	\end{equation*}
and the second inequality is due to the arithmetic-mean quadratic-mean inequality \cite[Exercise 2.14]{Steele2004}. The statement of the theorem now follows from the observation that the optimal value of~\eqref{ineq:am-qm} is identical to $\overline{C}_2(n,m)$.
\qed
\end{proof}

In the next proposition we derive a lower bound on $\overline{C}_1(n,m)$.

\begin{proposition}
\label{prop:wcwd-1}
For $d \geq n - 1$, the worst-case type-1 Wasserstein distance satisfies $\overline{C}_1(n,m) \geq \sqrt{\frac{(n-m)(n-m+1)}{n(n-1)}}$.
\end{proposition}

\begin{proof}
Assume first that $d = n$, and consider the empirical distribution $\hat{\mathbb{P}}_n$ with scenarios defined as in~\eqref{eq:wcwd_in_Rn}. Let $\{ I_j \} \in \mathfrak{P}(I,m)$ be an arbitrary $m$-set partition of~$I$ and note that $\text{gmed}(I_j)=\text{mean}(I_j)$ for every $j\in J$ due to the permutation symmetry of the $\bm{\xi}_i$. This is indeed the case because $\bm{0} \in \partial f_j (\text{mean}(I_j))$ for each $f_j (\bm{\zeta}) = \sum_{i \in I_j} \left\Vert y\mathbf{e} + (x-y)\mathbf{e}_i - \bm{\zeta} \right\Vert_2$, $j \in J$. Thus, we have 
    \begin{equation*}
\begin{aligned}
&\left\Vert \bm\xi_i - \text{mean}(I_j) \right\Vert_2 \\
& \quad =\ \left[ \left( x - \frac{x + (\vert I_j \vert-1)y}{\vert I_j \vert}\right)^2 + 
         (\vert I_j \vert-1) \left( y - \frac{x + (\vert I_j \vert-1)y}{\vert I_j \vert}\right)^2 \right]^{1/2} \\[2mm]
& \quad =\ \left[ \, \frac{\vert I_j \vert - 1}{\vert I_j \vert} \, \right]^{1/2} \, (x - y) \
        =\ \left[ \, \frac{n (\vert I_j \vert - 1)}{(n-1)\vert I_j \vert} \, \right]^{1/2}  \quad
\forall i \in I_j,
\end{aligned}
\end{equation*}
where the last equality follows from the definitions of $x$ and $y$ in~\eqref{eq:wcwd_in_Rn}. By Theorem~\ref{thm:partition} and Remark~\ref{rem:center} we therefore obtain
    \begin{equation*}
    \begin{aligned}
    C_1(\hat{\mathbb{P}}_n, m)\; &= \;
        \min_{\{ I_j \} \in \mathfrak{P}(I,m)} \frac{1}{n} \sum_{j \in J} \sum_{i \in I_j} \left\Vert \bm\xi_i - \text{mean}(I_j) \right\Vert_2 \\
        &= \;
        \min_{\{ I_j \} \in \mathfrak{P}(I,m)} \frac{1}{\sqrt{n(n-1)}} \sum_{j \in J} \sqrt{\vert I_j \vert(\vert I_j \vert - 1)}.
    \end{aligned}
\end{equation*} 
    By introducing auxiliary variables $z_j=\vert I_j \vert  - 1\in\mathbb{N}_0$, $j\in J$, we find that determining $C_1(\hat{\mathbb{P}}_n, m)$ is tantamount to solving
    \begin{equation*}
    \begin{aligned}
    C_1(\hat{\mathbb{P}}_n, m)\; &= \; 
        \frac{1}{\sqrt{n(n-1)}} \; \min_{\{z_j\} \subseteq \mathbb{N}_0} \left\{ \sum_{j \in J} \sqrt{z_j (z_j + 1)} :~ \sum_{j \in J} z_j = n - m\right\}.
\end{aligned}
\end{equation*}
    Observe that the objective function of $z_1 = n-m$ and $z_2 = \hdots = z_m = 0$ evaluates to $\sqrt{(n-m)(n-m+1)}$, which implies that $C_1(\hat{\mathbb{P}}_n, m) \leq \sqrt{\frac{(n-m)(n-m+1)}{n(n-1)}}$. Hence, it remains to establish the reverse inequality. To this end, we note that
    \begin{equation*}
    \begin{aligned}
    \sum_{j \in J} \sqrt{z_j (z_j + 1)} \; &=\;
        \Bigg[ \sum_{j \in J} z_j (z_j + 1) + \sum_{\substack{ j, j^\prime \in J \\ j \neq j^\prime }} \sqrt{z_j z_{j^\prime} (1+z_j)(1+z_{j^\prime})}\Bigg]^{1/2} \\
        &\geq \; \Bigg[ \sum_{j \in J} z_j (z_j + 1) + \sum_{\substack{ j, j^\prime \in J \\ j \neq j^\prime }} z_j z_{j^\prime} \Bigg]^{1/2} \\
        &= \; \left[ \Bigg( \sum_{j \in J} z_j \Bigg)^2 + \sum_{j \in J} z_j \right]^{1/2} \; = \; \sqrt{(n-m)(n-m+1)},
    \end{aligned}
    \end{equation*}
    and thus the claim follows for $d=n$. The cases $d=n-1$ and $d>n$ can be reduced to the case $d=n$ as in Proposition~\ref{prop:wcwd-2}. Details are omitted for brevity.
    \qed
\end{proof}

Proposition~\ref{prop:wcwd-1} asserts that $\overline{C}_1(n,m) \gtrsim \frac{n-m}{n-1}=\overline{C}{}^2_2(n,m)$ whenever $d\geq n-1$. Together with Theorem~\ref{thm:wcwd-1}, we thus obtain the following relation between the worst-case Wasserstein distances of types $l = 1$ and $l = 2$: 
\[ 
    \overline{C}^2_2(n,m) \leq \overline{C}_1(n,m) \leq \overline{C}_2(n,m).
\] 
We conjecture that the lower bound is tighter, but we were not able to prove this.

\subsection{Discussion}
\label{subsec:discuss}

Theorems~\ref{thm:wcwd-2} and~\ref{thm:wcwd-1} imply that $\overline{C}_l(n, m) \lesssim \sqrt{1 - p}$ for large~$n$ and for $l\in\{1,2\}$, where $p=\frac{m}{n}$ represents the desired reduction factor. 
The significance of this result is that it offers {\em a priori} guidelines for selecting the number $m$ of support points in the reduced distribution. To see this, consider any empirical distribution $\hat{\mathbb{P}}_n = \frac{1}{n} \sum_{i\in I} \delta_{\bm{\xi}_i}$, and denote by $r\geq 0$ and $\bm{\mu}\in\mathbb{R}^d$ the radius and the center of any (ideally the smallest) ball enclosing $\bm{\xi}_1,\ldots,\bm{\xi}_n$, respectively. In this case, we have
\begin{equation}
	\label{eq:approximation-quality}
	C_l(\hat{\mathbb{P}}_n,m)
	=r\cdot C_l\left(\frac{1}{n} \sum_{i\in I} \delta_{\frac{\bm{\xi}_i-\bm{\mu}}{r}},m\right)\leq r\cdot \overline{C}_l(n, m)\lesssim r\cdot \sqrt{1-p},
\end{equation}
where the inequality holds because $\|(\bm{\xi}_i-\bm{\mu})/r\|_2\leq 1$ for every $i\in I$. 
Note that \eqref{eq:approximation-quality} enables us to find an upper bound on the smallest $m$ guaranteeing that $C_l(\hat{\mathbb{P}}_n,m)$ falls below a prescribed threshold ({\em i.e.}, guaranteeing that the reduced $m$-point distribution remains within some prescribed distance from $\hat{\mathbb{P}}_n$).

Even though the inequality in~\eqref{eq:approximation-quality} can be tight, which has been established in Proposition~\ref{prop:wcwd-2}, one might suspect that typically $C_l(\hat{\mathbb{P}}_n,m)$ is significantly smaller than $r\cdot \sqrt{1-p}$ when the points $\bm{\xi}_i\in\mathbb{R}^d$, $i\in I$, are sampled randomly from a standard distribution, {\em e.g.}, a multivariate uniform or normal distribution. However, while the upper bound~\eqref{eq:approximation-quality} can be loose for low-dimensional data, Proposition~\ref{prop:wcwd-normal} below suggests that it is surprisingly tight in high dimensions---at least for $l=2$.


\begin{proposition}
\label{prop:wcwd-normal}
For any $\epsilon>0$ and $\delta>0$ there exist $c>0$ and $d\in\mathbb{N}$ such that
\begin{equation}
\mathbb{P}^n\left( \|\bm{\xi}_i\|_2\leq 1~\forall i\in I \text{ and }C_2\left( \frac{1}{n} \sum_{i\in I} \delta_{\bm{\xi}_i}, m\right) \geq \sqrt{1-p}-\delta\right)\geq 1-\epsilon,
\end{equation}
where $p=\frac{m}{n}$, and the support points $\bm{\xi}_i$, $i\in I$, are sampled independently from the normal distribution $\mathbb{P}$ with mean $\bm{0}\in\mathbb{R}^d$ and covariance matrix $(\sqrt{d-1}+c)^{-2}\mathbb{I}\in\mathbb{S}^d$.
\end{proposition}

Proposition~\ref{prop:wcwd-normal} can be paraphrased as follows. Sampling the $\bm{\xi}_i$, $i\in I$, independently from the normal distribution $\mathbb{P}$ yields a (random) empirical distribution~$\hat{\mathbb{P}}_n$ that is feasible and $\delta$-suboptimal in~\eqref{opt:wcwd} with probability $1-\epsilon$. The intuition behind this result is that, in high dimensions, samples drawn from $\mathbb{P}$ are almost orthogonal and close to the surface of the unit ball with high probability. Indeed, these two properties are shared by the worst case distribution~\eqref{eq:wcwd_in_Rn} in high dimensions.

\begin{proof} \emph{of Proposition~\ref{prop:wcwd-normal}} $\,$
Theorem~\ref{thm:partition} and Remark~\ref{rem:center} imply that 
\begin{equation}
	\label{eq:aux-c2}
	C_2^2\left( \frac{1}{n} \sum_{i\in I} \delta_{\bm{\xi}_i}, m\right) = \min_{\{ I_j \} \in \mathfrak{P}(I,m)} \, \frac{1}{n} \sum_{j \in J} \sum_{i \in I_j} 
	\left\Vert \bm\xi_i - \text{mean}(I_j) \right\Vert_2^2 .
\end{equation}
From the proof of Theorem~\ref{thm:wcwd-2} we further know that \eqref{eq:aux-c2} can be expressed as a continuous function $f(\mathbf{S})$ of the Gram matrix $\mathbf{S} = [\bm\xi_1, \hdots, \bm\xi_n]^\top [\bm\xi_1, \hdots, \bm\xi_n]$, that is,
\begin{equation*}
	f(\mathbf{S})=\min_{\{ I_j \} \in \mathfrak{P}(I,m)} \,   \frac{1}{n} \sum_{j \in J}\frac{1}{\vert I_j \vert^2} \sum_{i \in I_j} \Bigg( \vert I_j \vert^2 s_{ii} - 2\vert I_j \vert \sum_{k \in I_j} s_{ik}
	+ \sum_{k \in I_j} s_{kk} + \sum_{\substack{k,k^\prime \in I_j \\ k \neq k^\prime}} s_{kk^\prime} \Bigg).
\end{equation*}
An elementary calculation shows that $f(\mathbb{I})=\frac{n-m}{n}=1-p$. Thus, by the continuity of $f(\cdot)$, there exists $\eta\in (0,1)$ with $\sqrt{f(\mathbf{S})}\geq \sqrt{1-p}-\delta$ whenever $\|\mathbf{S}-\mathbb{I}\|_{\rm max}\leq \eta$.

We are now ready to construct $\mathbb{P}$. First, select $c>0$ large enough to ensure~that
\[
\left( 1-\frac{4}{c^2} e^{-\frac{c^2}{4}}\right)^n\geq 1-\frac{\epsilon}{2}.
\]
Then, select $d\in\mathbb{N}$ large enough such that
\[
	\frac{\sqrt{d-1}-c}{\sqrt{d-1}+c} \geq 1-\eta\quad\text{and}\quad n(n-1) \mathrm{\Phi}(-\eta(\sqrt{d-1}+c))\leq \frac{\epsilon}{2},
\]
where $\mathrm{\Phi}(\cdot)$ denotes the univariate standard normal distribution function. Observe that the distribution $\mathbb{P}$ is completely determined by $c$ and $d$. Next, we find that
\begin{align}
	&\mathbb{P}^n\left( \|\bm{\xi}_i\|_2\leq 1~\forall i \text{ and }C_2\left( \frac{1}{n} \sum_{i\in I} \delta_{\bm{\xi}_i}, m\right)\geq \sqrt{1-p}-\delta\right) \nonumber \\
	\geq ~&\mathbb{P}^n\left(  \|\bm{\xi}_i\|_2\leq 1~\forall i \text{ and }\|\mathbf{S}-\mathbb{I}\|_{\rm max}\leq \eta \right) \nonumber \\
	= ~&\mathbb{P}^n\left( 1-\eta \leq \|\bm{\xi}_i\|_2\leq 1~\forall i \text{ and }  | \bm\xi_i^\top \bm\xi_j | \leq \eta ~\forall i\neq j \right) \nonumber \\
	\geq ~&\mathbb{P}^n\left( \frac{\sqrt{d-1}-c}{\sqrt{d-1}+c} \leq \|\bm{\xi}_i\|_2\leq 1~\forall i \text{ and } 
	| \bm\xi_i^\top \bm\xi_j | \leq \eta\cdot \| \bm{\xi}_j\|_2 ~\forall i\neq j \right) \nonumber \\
	\geq ~&\mathbb{P}^n\left( \frac{\sqrt{d-1}-c}{\sqrt{d-1}+c} \leq \|\bm{\xi}_i\|_2\leq 1~\forall i \right) + \mathbb{P}^n\left(| \bm\xi_i^\top \bm\xi_j | \leq \eta \cdot \| \bm{\xi}_j\|_2 ~\forall i\neq j \right)-1, \label{eq:probability}
\end{align}
where the first and second inequalities follow from the construction of $\eta$ and $c$, respectively, while the third inequality exploits the union bound. We also have
\begin{align}
	\nonumber
	\mathbb{P}^n\left( \frac{\sqrt{d-1}-c}{\sqrt{d-1}+c} \leq \|\bm{\xi}_i\|_2\leq 1~\forall i\right)&=\mathbb{P}\left( \frac{\sqrt{d-1}-c}{\sqrt{d-1}+c} \leq \|\bm{\xi}_1\|_2\leq 1\right)^n\\
	 & \geq \left( 1-\frac{4}{c^2} e^{-\frac{c^2}{4}}\right)^n\geq 1-\frac{\epsilon}{2},
	 \label{eq:cmu-result}
\end{align}
where the equality holds due to the independence of the $\bm{\xi}_i$, while the first and second inequalities follow from Lemma~2.8 in~\cite{Hopcroft2012} and the construction of~$c$, respectively. By the choice of $d$, we finally obtain
\begin{align}
	\nonumber
	\mathbb{P}^n\left(| \bm\xi_i^\top \bm\xi_j | \leq \eta \cdot \| \bm{\xi}_j\|_2 ~\forall i\neq j \right)&
	\geq 1-\sum_{i\neq j}\mathbb{P}^n\left(| \bm\xi_i^\top \bm\xi_j | \geq \eta \cdot \| \bm{\xi}_j\|_2\right)\\
	 & = 1-n(n-1) \mathrm{\Phi}(-\eta(\sqrt{d-1}+c))\geq 1-\frac{\epsilon}{2}.
	 \label{eq:i-not-j}
\end{align}
The equality in~\eqref{eq:i-not-j} holds due to the rotation symmetry of~$\mathbb{P}$, which implies that
\begin{align}
	\nonumber
	\mathbb{P}^n\left(| \bm\xi_i^\top \bm\xi_j | \geq \eta \cdot \| \bm{\xi}_j\|_2 \right)&= \mathbb{P}\left(| \bm\xi_1^\top \mathbf{e}_1 | \geq \eta \right) =2 \mathrm{\Phi}(-\eta(\sqrt{d-1}+c)).
\end{align}
The claim then follows by substituting~\eqref{eq:cmu-result} and~\eqref{eq:i-not-j} into~\eqref{eq:probability}. \qed
\end{proof}

\begin{figure}[h]
 	\centering
	\includegraphics[width=0.28\paperwidth]{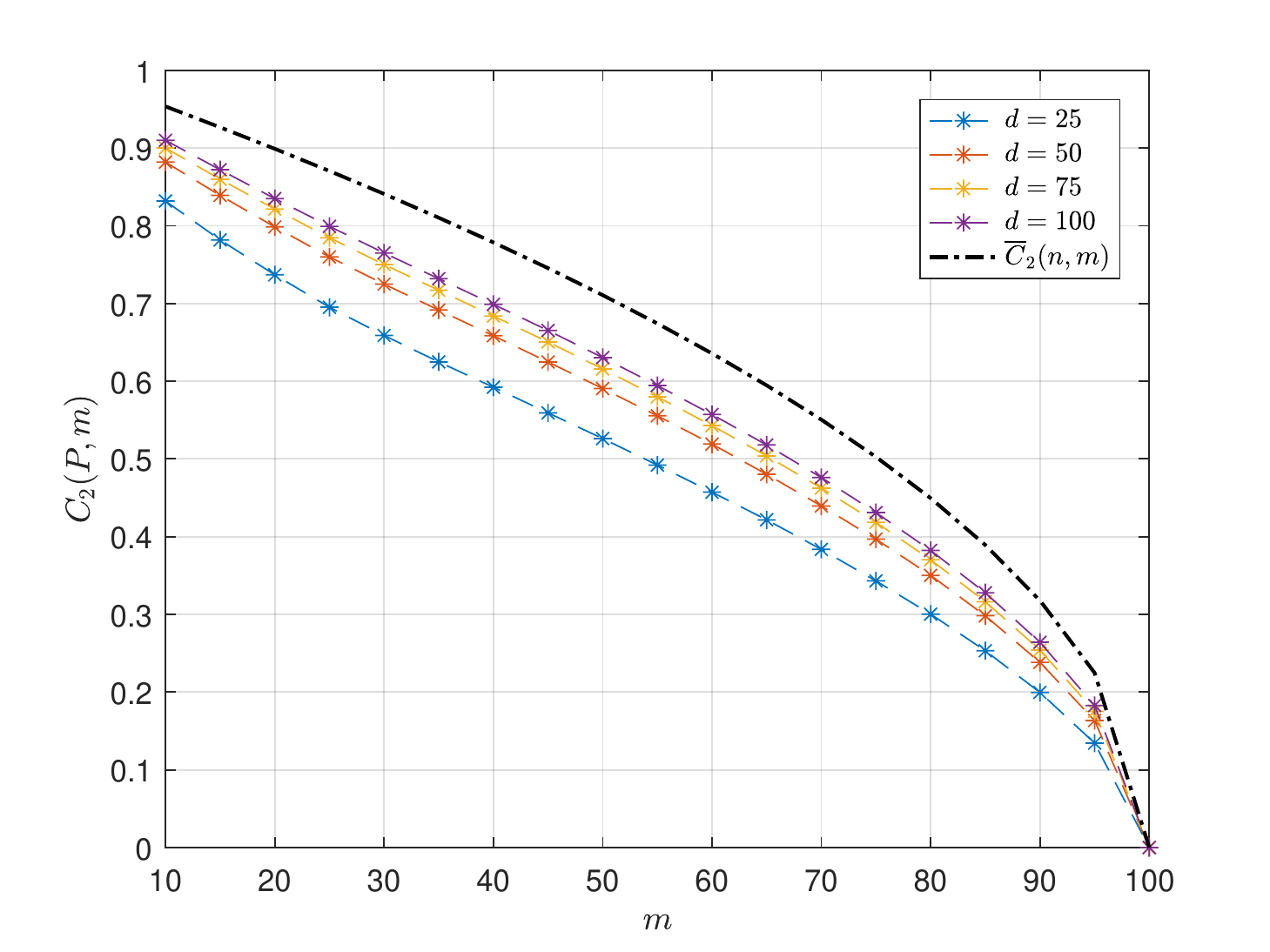}
	\includegraphics[width=0.28\paperwidth]{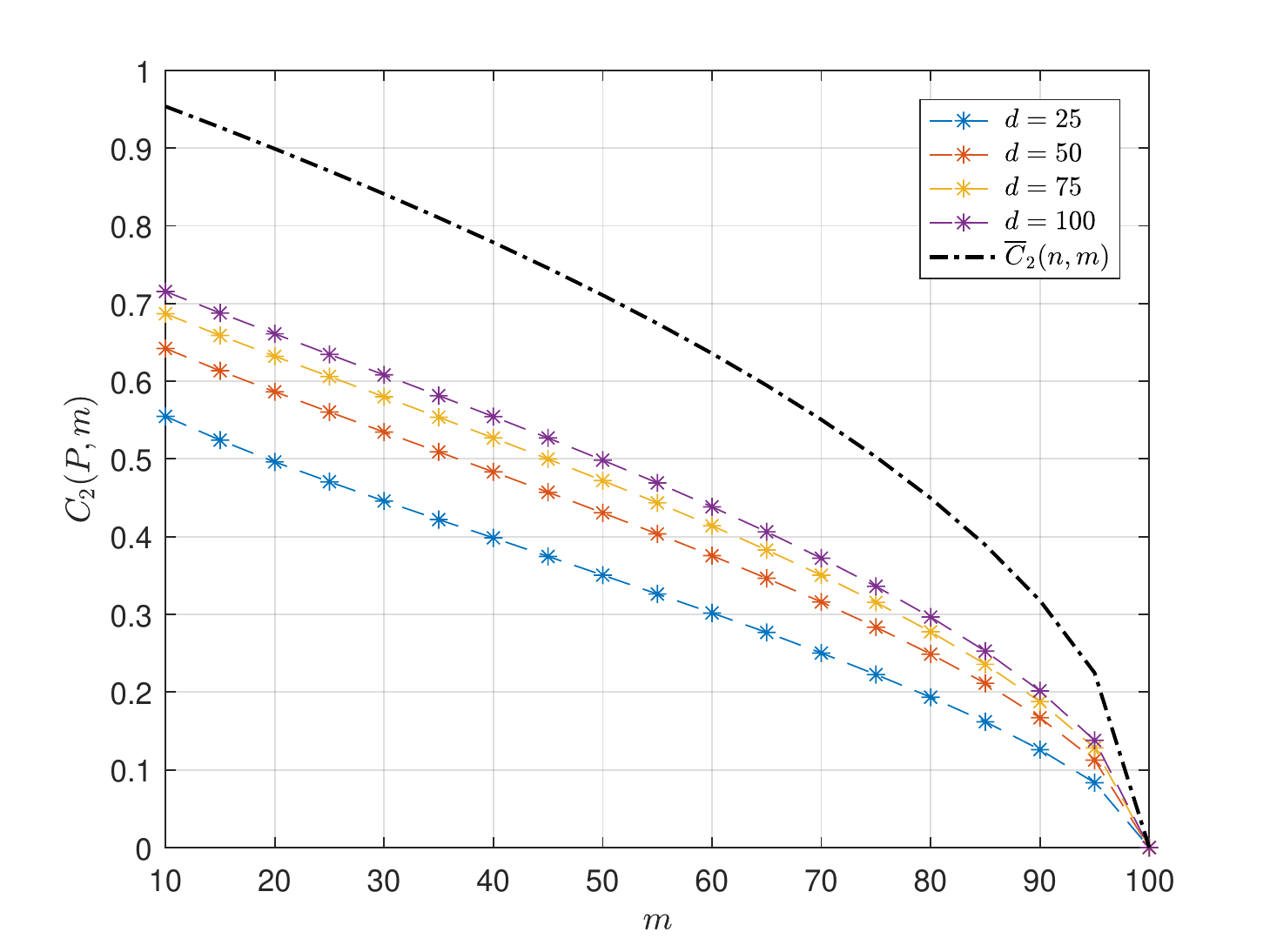}
	\caption{Comparison between $C_2(\hat{\mathbb{P}}_n,m)$ and $\overline{C}_2(n,m)$ under uniform (left panel) and normal (right panel) sampling.}
	\label{fig:wc_bound}
\end{figure}

Figure~\ref{fig:wc_bound} compares $\overline{C}_2(m,n)$ with $C_2(\hat{\mathbb{P}}_n, m)$ for $n=100$, $m \in\{ 10, \hdots, 100\}$ and $d\in\{ 25, 50, 75, 100\}$. The $n$ original support points are sampled randomly from the uniform distribution on the unit ball (left panel) and the normal distribution from Proposition~\ref{prop:wcwd-normal} with $c=2.97$, which ensures that $\Vert \bm\xi_i \Vert_2 \leq 1$ with 95\% probability (right panel). Note that $C_2(\hat{\mathbb{P}}_n, m)$ is random. Thus, all shown values are averaged across 100 independent trials. Figure~\ref{fig:wc_bound} confirms that $C_2(\hat{\mathbb{P}}_n,m)$ approaches the worst-case bound $\overline{C}_2(n,m)$ as the dimension $d$ increases.

\section{Guarantees for Discrete Scenario Reduction}
\label{section:compare}

For $n$-point empirical distributions $\hat{\mathbb{P}}_n = \frac{1}{n} \sum_{i=1}^n \delta_{\bm{\xi}_i}$ supported on $\mathbb{R}^d$, we now study the loss of optimality incurred by solving the discrete scenario reduction problem instead of its continuous counterpart. More precisely, we want to determine the point-wise largest lower bound $\underline{\kappa}_l (n, m)$ and the point-wise smallest upper bound $\overline{\kappa}_l (n, m)$ that satisfy
\begin{equation}
\label{eq:ratio}
	\underline{\kappa}_l(n,m) \cdot C_l(\hat{\mathbb{P}}_n,m) \; \leq \; 
    D_l(\hat{\mathbb{P}}_n,m) \; \leq \; 
    \overline{\kappa}_l(n,m) \cdot C_l(\hat{\mathbb{P}}_n,m) \quad
    \forall \, \hat{\mathbb{P}}_n \in \mathcal{P}_{\mathrm{E}}(\mathbb{R}^d, n)
\end{equation}
for the Wasserstein distances of type $l \in \{ 1, 2 \}$. Note that the existence of finite bounds $\underline{\kappa}_l(n,m)$ and $\overline{\kappa}_l(n,m)$ is not a priori obvious as they do not depend on the dimension $d$. Moreover, while it is clear that $\underline{\kappa}_l(n,m) \geq 1$ if it exists, it does not seem easy to derive a na\"ive upper bound on $\overline{\kappa}_l(n,m)$.

Our derivations in this section will use the following result, which is the analogue of Theorem~\ref{thm:partition} for the discrete scenario reduction problem.

\begin{theorem}
\label{thm:partition-red}
For any type-$l$ Wasserstein distance induced by any norm $\Vert \cdot \Vert$, the discrete scenario reduction problem can be reformulated as
\begin{equation*}
	D_l(\hat{\mathbb{P}}_n, m) = \min_{\{ I_j \} \in \mathfrak{P}(I,m)} \, \left[ \frac{1}{n} \sum_{j \in J} \min_{\bm\zeta_j \in \{ \bm\xi_i:\, i \in I_j \}} \sum_{i \in I_j} \Vert \bm\xi_i - \bm\zeta_j \Vert^l \right]^{1/l}.
\end{equation*}
\end{theorem}

\begin{proof} 
The proof is similar to the proof of Theorem~\ref{thm:partition} and is therefore omitted. \qed
\end{proof}

The remainder of this section derives lower and upper bounds on $\underline{\kappa}_l(n,m)$ and $\overline{\kappa}_l(n,m)$ for Wasserstein distances of type $l = 2$ (Section~\ref{section:compare-2}) and $l = 1$ (Section~\ref{section:compare-1}), respectively. To eliminate trivial cases, we assume throughout this section that $n \geq 2$, $m \in \{ 1, \ldots, n - 1 \}$ and $d \geq 2$.


\subsection{Guarantees for the Type-2 Wasserstein Distance}
\label{section:compare-2}


We first bound $\overline{\kappa}_2 (n,m)$ in equation~\eqref{eq:ratio} from above (Theorem~\ref{thm:red2-sqrt(2)}) and below (Proposition~\ref{prop:red2-sqrt(2)}).

\begin{theorem}
\label{thm:red2-sqrt(2)}
The upper bound $\overline{\kappa}_2 (n,m)$ in~\eqref{eq:ratio} satisfies $\overline{\kappa}_2 (n,m) \leq \sqrt{2}$ for all $n,m$.
\end{theorem}

\begin{proof} 
The proof proceeds in two steps. We first show that $\overline{\kappa}_2(n,m) \leq \sqrt{2}$ for all $n$ when $m = 1$ (Step~1). Then we extend the result to all $n$ and $m$ (Step~2).

\paragraph{Step 1:}
Fix any $\hat{\mathbb{P}}_n \in \mathcal{P}_\mathrm{E}(\mathbb{R}^d, n)$. W.l.o.g., we can assume that $\text{mean}(I) = \bm{0}$ and $\frac{1}{n} \sum_{i \in I} \Vert \bm\xi_i \Vert_2^2 = 1$ by re-positioning and scaling the atoms $\bm\xi_i$ appropriately. Note that the re-positioning does not affect $C_2(\hat{\mathbb{P}}_n,1)$ or $D_2(\hat{\mathbb{P}}_n,1)$, and the positive homogeneity of the Wasserstein distance implies that the scaling affects both $C_2(\hat{\mathbb{P}}_n,1)$ and $D_2(\hat{\mathbb{P}}_n,1)$ in the same way and thus preserves their ratio $\overline{\kappa}_2 (n,1)$. Theorem~\ref{thm:partition} and Remark~\ref{rem:center} then imply that
	\begin{equation*}
		C_2(\hat{\mathbb{P}}_n,1) \; = \;
		\left[ \frac{1}{n} \sum_{i \in I} \Vert \bm\xi_i - \text{mean}(I) \Vert_2^2 \right]^{1/2} \; = \; 1.
	\end{equation*}
Step 1 is thus complete if we can show that $D_2(\hat{\mathbb{P}}_n,1) \leq \sqrt{2}$. Indeed, we have
	\begin{equation}
	\label{eq:(n,1)-sqrt(2)}
	\begin{aligned}
		D^2_2(\hat{\mathbb{P}}_n,1) \
		&=\ \min_{j \in I}\ \frac{1}{n} \sum_{i \in I} \Vert \bm\xi_i - \bm\xi_j \Vert_2^2 \
		=\min_{j \in I}\ \frac{1}{n} \sum_{i \in I} (\bm\xi_i - \bm\xi_j)^\top(\bm\xi_i - \bm\xi_j) \\
		&=\ \min_{j \in I}\ \frac{1}{n} \sum_{i \in I} \left( \bm\xi_i^\top \bm\xi_i - 2 \bm\xi_i^\top \bm\xi_j + \bm\xi_j^\top \bm\xi_j \right) \
		=\ \min_{j \in I}\ \frac{1}{n} \sum_{i \in I} \left( \bm\xi_i^\top \bm\xi_i + \bm\xi_j^\top \bm\xi_j \right) \\
		&=\ \min_{j \in I}\ \Vert \bm\xi_j \Vert_2^2 + \frac{1}{n} \sum_{i \in I} \Vert \bm\xi_i \Vert_2^2 \
		=\ 1 + \min_{j \in I}\ \Vert \bm\xi_j \Vert_2^2 \
		\leq\ 2 ,
	\end{aligned}
	\end{equation}
	where the first equality is due to Theorem~\ref{thm:partition-red}, the fourth follows from $\sum_{i \in I} \bm\xi_i = n \cdot \text{mean}(I) = \bm{0}$, and the inequality holds since $\min_{j \in I} \Vert \bm\xi_j \Vert_2^2 \leq  \frac{1}{n} \sum_{i \in I} \Vert \bm\xi_i \Vert_2^2 = 1$.
	
    \paragraph{Step 2:} 
    Fix any $\hat{\mathbb{P}}_n \in \mathcal{P}_\mathrm{E}(\mathbb{R}^d, n)$. Theorem~\ref{thm:partition} and Remark~\ref{rem:center} imply that
	\begin{equation*}
	\begin{aligned}
		C_2(\hat{\mathbb{P}}_n,m) =
		\min_{\{ I_j\} \, \in \, \mathfrak{P}(I,m)} \left[ \frac{1}{n} \sum_{j \in J} \sum_{i \in I_j} \Vert \bm\xi_i - \text{mean}(I_j) \Vert_2^2 \right]^{1/2}.
	\end{aligned}
	\end{equation*}
	For an optimal partition $\{ I^\star_j \}$ to this problem, $C_2(\hat{\mathbb{P}}_n,m)$ can be expressed as
	\begin{equation*}
		C_2(\hat{\mathbb{P}}_n,m) = \left[ \sum_{j \in J} \frac{\vert I^\star_j \vert}{n} C^2_{2,j} \right]^{1/2}\!\!\!\!\! \quad \text{with}~~
		C_{2,j} = \left[ \frac{1}{\vert I^\star_j \vert} \sum_{i \in I^\star_j} \Vert \bm\xi_i - \text{mean}(I^\star_j) \Vert_2^2 \right]^{1/2}.
	\end{equation*}
	From our discussion in Step~1 we know that $C_{2,j}$ represents the type-$2$ Wasserstein distance between the conditional empirical distribution $\hat{\mathbb{P}}_n^{\,j} = \frac{1}{\vert I^\star_j \vert} \sum_{i \in I^\star_j} \delta_{\bm\xi_i}$ and its closest Dirac distribution, that is, $C_2(\hat{\mathbb{P}}_n^{\,j}, 1)$. Analogously, we obtain that
	\begin{equation*}
	\begin{aligned}
		D_2(\hat{\mathbb{P}}_n,m) \; &\leq \;
        \left[ \sum_{j \in J} \frac{\vert I^\star_j \vert}{n} D^2_{2,j} \right]^{1/2}\!\!\!\!\! \quad \text{with}~~
		D_{2,j} = \left[ \min_{j \in I^\star_j} \frac{1}{\vert I^\star_j \vert} \sum_{i \in I^\star_j} \Vert \bm\xi_i - \bm\xi_j \Vert_2^2 \right]^{1/2} \\
         &\leq \; \left[ \sum_{j \in J} \frac{\vert I^\star_j \vert}{n} (2C^2_{2,j}) \right]^{1/2} \!\!\! = \ \
        \sqrt{2} \, C_2(\hat{\mathbb{P}}_n,m),
	\end{aligned}
	\end{equation*}
	where the first inequality holds since the optimal partition $\{ I^\star_j \}$ for $C_2(\hat{\mathbb{P}}_n,m)$ is typically suboptimal in $D_2(\hat{\mathbb{P}}_n,m)$, the second inequality follows from the fact that $D_{2,j} = D_2(\hat{\mathbb{P}}_n^{\,j}, 1)$ and $D_2(\hat{\mathbb{P}}_n^{\,j}, 1) \leq \sqrt{2} C_2(\hat{\mathbb{P}}_n^{\,j}, 1)$ due to Step~1, and the identity follows from the definition of $C_{2,j}$. The statement now follows.
	\qed
\end{proof}

\begin{figure}[h]
 	\centering
	\includegraphics[width=0.28\paperwidth]{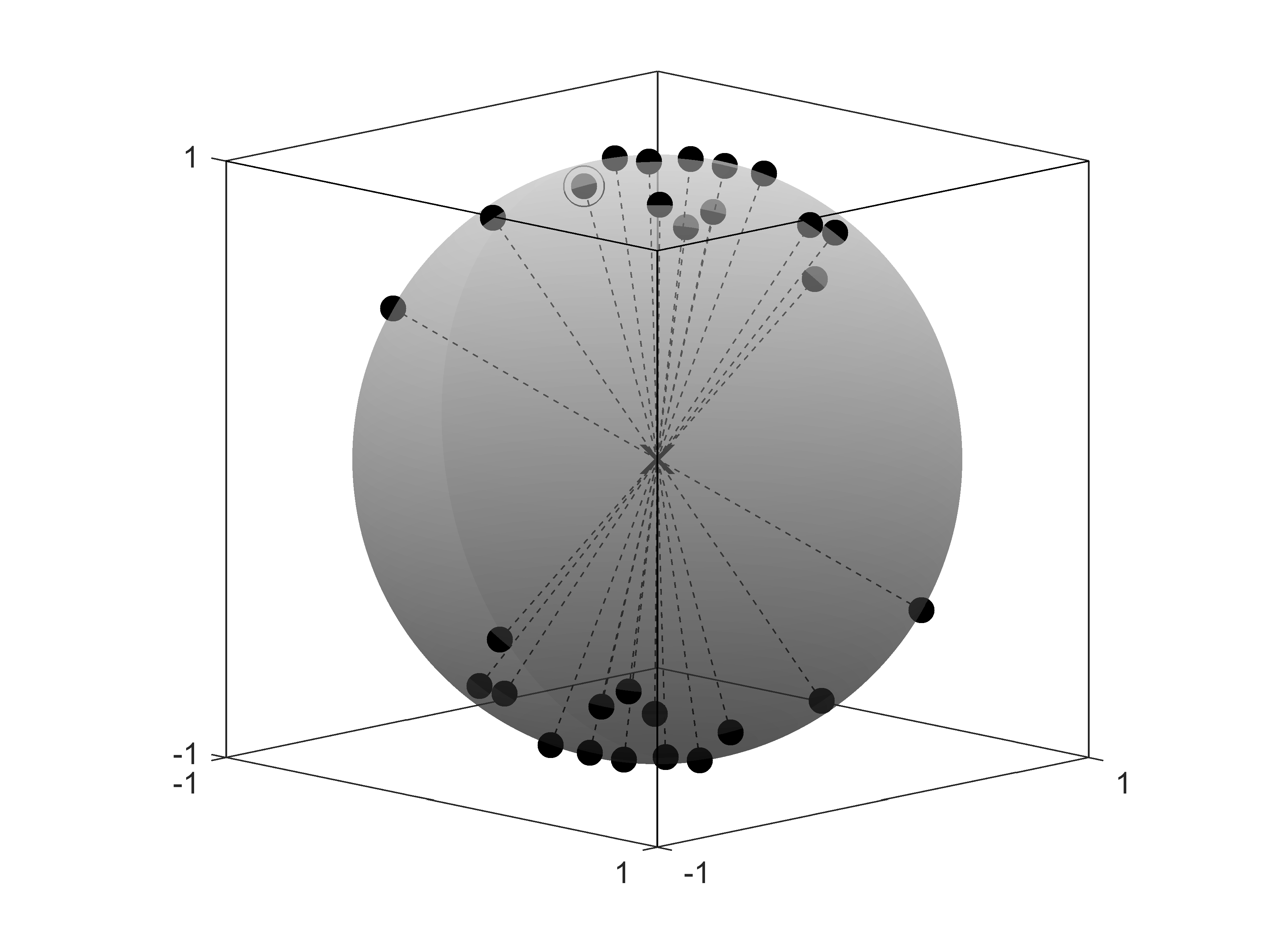}
	\includegraphics[width=0.28\paperwidth]{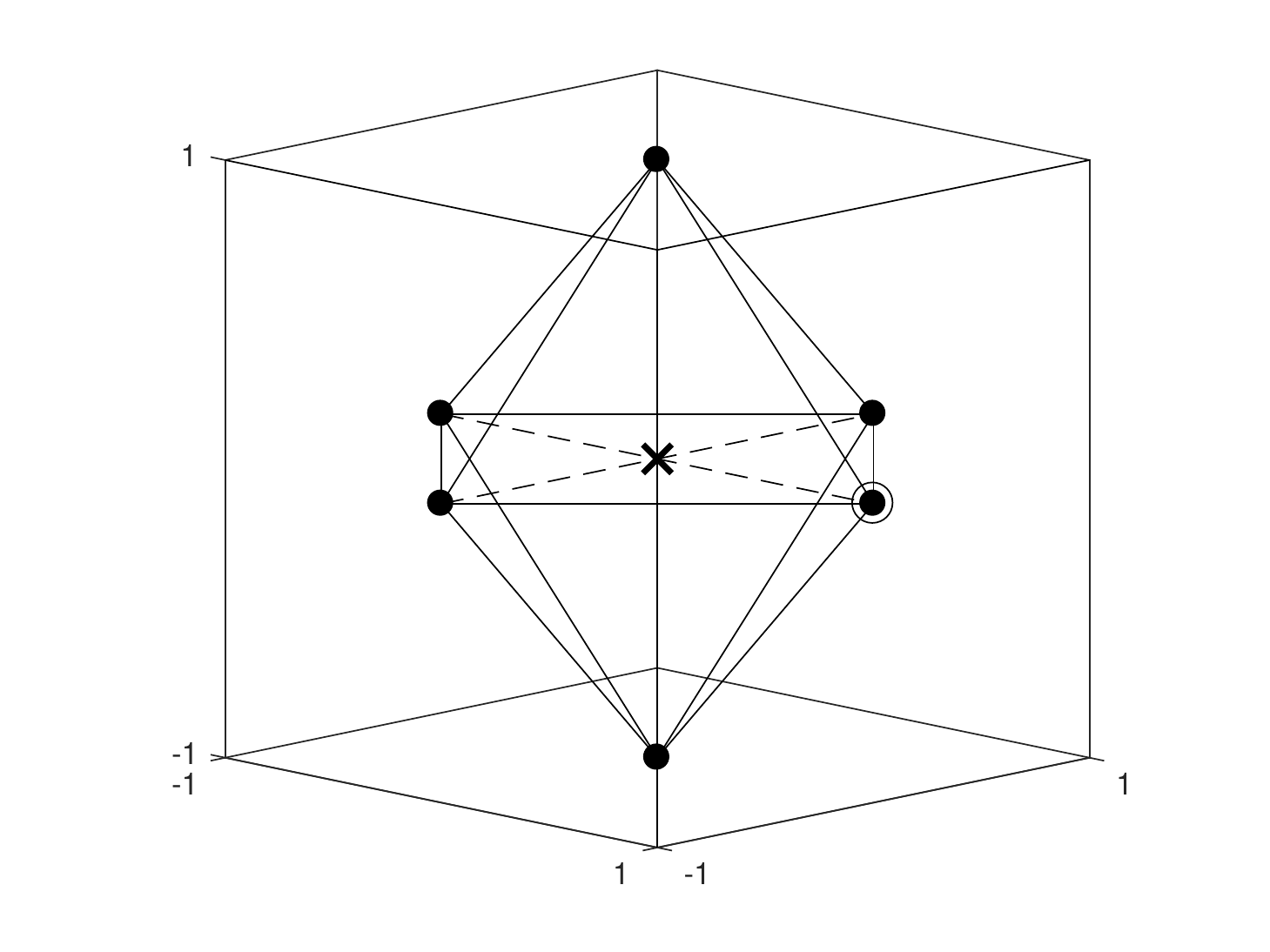}
	\caption{Empirical distributions in $\mathbb{R}^3$ that maximize the ratio between $D_l(\hat{\mathbb{P}}_n,1)$ and $C_l(\hat{\mathbb{P}}_n,1)$ for $l = 2$ and $\Vert \cdot \Vert = \Vert \cdot \Vert_2$ (left panel) as well as $l = 1$ and $\Vert \cdot \Vert = \Vert \cdot \Vert_1$ (right panel). In both cases, the continuous scenario reduction problem is optimized by the Dirac distribution at $\bm{0}$ (marked as $\times$), whereas the discrete scenario reduction problem is optimized by any of the atoms (such as $\circ$).}
	\label{fig:wsss_2_and_1}
\end{figure}

\begin{proposition}
\label{prop:red2-sqrt(2)}
There is $\hat{\mathbb{P}}_n \in \mathcal{P}_\mathrm{E} (\mathbb{R}^d, n)$ with $D_2(\hat{\mathbb{P}}_n,m) = \sqrt{2} C_2(\hat{\mathbb{P}}_n,m)$ for all~$n,m$.
\end{proposition}

\begin{proof} 
In analogy to the proof of Theorem~\ref{thm:red2-sqrt(2)}, we first show the statement for $m = 1$ (Step~1) and then extend the result to $m > 1$ (Step~2).

\paragraph{Step 1:}
The first step in the proof of Theorem~\ref{thm:red2-sqrt(2)} shows that $\hat{\mathbb{P}}_n \in \mathcal{P}_\mathrm{E} (\mathbb{R}^d, n)$ satisfies $D_2(\hat{\mathbb{P}}_n,1) = \sqrt{2} C_2(\hat{\mathbb{P}}_n,1)$ if $\sum_{i \in I} \bm\xi_i = \bm{0}$ and $\Vert \bm\xi_1 \Vert = \hdots = \Vert \bm\xi_n \Vert = 1$. For an even number $n = 2k$, $k \in \mathbb{N}$, both conditions are satisfied if we place $\bm\xi_1, \hdots, \bm\xi_k$ on the surface of the unit ball in $\mathbb{R}^d$ and then choose $\bm\xi_{k+i} = -\bm\xi_i$ for $i = 1, \hdots, k$ (see left panel of Figure~\ref{fig:wsss_2_and_1} for an illustration in $\mathbb{R}^3$). Likewise, for an odd number $n = 2k + 3$, $k \in \mathbb{N}_0$, we can place $\bm\xi_1, \hdots, \bm\xi_k$ on the surface of the unit ball, choose $\bm\xi_{k+i} = -\bm\xi_i$ for $i = 1, \hdots, k$ and fix $\bm\xi_{2k+1} = \mathbf{e}_1$, $\bm\xi_{2k+2} = -\frac{1}{2} \mathbf{e}_1 + \frac{\sqrt{3}}{2} \mathbf{e}_2$ and $\bm\xi_{2k+3} = -\frac{1}{2} \mathbf{e}_1 -\frac{\sqrt{3}}{2} \mathbf{e}_2$.

\paragraph{Step 2:}
To prove the statement for $m > 1$, we construct an empirical distribution $\hat{\mathbb{P}}_n \in \mathcal{P}_\mathrm{E} (\mathbb{R}^d, n)$ whose atoms satisfy $\text{supp} (\hat{\mathbb{P}}_n) = \mathrm{\Xi}_1 \cup \mathrm{\Xi}_2$ with $| \mathrm{\Xi}_1 | = n - m + 1$ and $| \mathrm{\Xi}_2 | = m - 1$. The atoms $\bm{\xi}_1, \ldots, \bm{\xi}_{n-m+1}$ in $\mathrm{\Xi}_1$ are selected according to the recipe outlined in Step~1, whereas the atoms $\bm{\xi}_{n-m+2}, \ldots, \bm{\xi}_n$ in $\mathrm{\Xi}_2$ satisfy $\bm{\xi}_{n-m+1+i} = (1+iM) \mathbf{e}_1$, $i = 1, \ldots, m - 1$, for any number $M$ satisfying $M>2\,\sqrt{n-m+1}$. A direct calculation then shows that the atoms in $\mathrm{\Xi}_2$ are sufficiently far away from those in $\mathrm{\Xi}_1$ as well as from each other so that any optimal partition $\{ I_j^\star \}$ to the discrete scenario reduction problem in Theorem~\ref{thm:partition-red} as well as the continuous scenario reduction problem in Theorem~\ref{thm:partition} consists of the sets $\{ i : \bm\xi_i \in \mathrm{\Xi}_1 \}$ and $\{ i \}$, $\bm\xi_i \in \mathrm{\Xi}_2$. The result then follows from the fact that either problem accumulates a Wasserstein distance of $0$ over the atoms in $\mathrm{\Xi}_2$, whereas the Wasserstein distance of $D_2(\hat{\mathbb{P}}_n,m)$ is a factor of $\sqrt{2}$ bigger than the Wasserstein distance of $C_2(\hat{\mathbb{P}}_n,m)$ over the atoms in $\mathrm{\Xi}_1$ (see Step~1).
\qed
\end{proof}


Theorem~\ref{thm:red2-sqrt(2)} and Proposition~\ref{prop:red2-sqrt(2)} imply that $\overline{\kappa}_2(n,m) = \sqrt{2}$ for all $n$ and $m$, that is, the bound is indeed \emph{independent} of both the number of atoms $n$ in the empirical distribution and the number of atoms $m$ in the reduced distribution. We now show that the na\"ive lower bound of $1$ on the approximation ratio is essentially tight.

\begin{proposition}
\label{prop:red2-1}
The lower bound $\underline{\kappa}_2 (n, m)$ in~\eqref{eq:ratio} satisfies $\underline{\kappa}_2 (n, m) = 1$ whenever $n \geq 3$ and $m \in \{1,\ldots,n-2\}$, while $\underline{\kappa}_2 (n, n-1) = \sqrt{2}$ always.
\end{proposition}

\begin{proof} 
We first prove $\underline{\kappa}_2 (n, m) = 1$ when $m = 1$ and $n\geq3$ (Step~1) and when $m \in \{ 2, \ldots, n - 2 \}$ (Step~2). Then, we show $\underline{\kappa}_2 (n, n - 1) = \sqrt{2}$ (Step~3).

\paragraph{Step 1:}
Choose $\hat{\mathbb{P}}_n \in \mathcal{P}_\mathrm{E} (\mathbb{R}^d, n)$ such that the first $n - 1$ atoms $\bm{\xi}_1, \ldots, \bm{\xi}_{n-1}$ are selected according to the recipe outlined in Step~1 in the proof of Proposition~\ref{prop:red2-sqrt(2)} and $\bm{\xi}_n = \bm{0}$. We thus have $\text{mean} (I) = \bm{0}$, and Theorem~\ref{thm:partition} and Remark~\ref{rem:center} imply that the optimal continuous scenario reduction is given by the Dirac distribution $\delta_{\bm{0}}$. Since $\bm{0} \in \text{supp} (\hat{\mathbb{P}}_n)$, we have $C_2(\hat{\mathbb{P}}_n,1) = D_2(\hat{\mathbb{P}}_n,1)$ and the result follows.

\paragraph{Step 2:}
To prove the statement for $m > 1$, we proceed as in Step~2 in the proof of Proposition~\ref{prop:red2-sqrt(2)}. In particular, we construct an empirical distribution $\hat{\mathbb{P}}_n \in \mathcal{P}_\mathrm{E} (\mathbb{R}^d, n)$ whose atoms satisfy $\text{supp} (\hat{\mathbb{P}}_n) = \mathrm{\Xi}_1 \cup \mathrm{\Xi}_2$ with $| \mathrm{\Xi}_1 | = n - m + 1$ and $| \mathrm{\Xi}_2 | = m - 1$. The atoms $\bm{\xi}_1, \ldots, \bm{\xi}_{m-m+1}$ in $\mathrm{\Xi}_1$ are selected according to the recipe outlined in Step~1 of this proof, whereas the remaining atoms $\bm{\xi}_{n-m+1}, \ldots, \bm{\xi}_n$ in $\mathrm{\Xi}_2$ satisfy $\bm{\xi}_{n-m+1+i} = (1+iM) \mathbf{e}_1$, $i = 1, \ldots, m - 1$, for any $M>2 \, \sqrt{n-m+1}$. A similar argument as in the proof of Proposition~\ref{prop:red2-sqrt(2)} then shows that $C_2(\hat{\mathbb{P}}_n,m) = D_2(\hat{\mathbb{P}}_n,m)$.

\paragraph{Step 3:}
Fix any $\hat{\mathbb{P}}_n \in \mathcal{P}_\mathrm{E} (\mathbb{R}^d, n)$. W.l.o.g., assume that $\{ \bm{\xi}_{n-1}, \bm{\xi}_n \}$ is the closest pair of atoms in terms of Euclidean distance, and let $d_{\text{min}} = \Vert \bm\xi_{n} - \bm\xi_{n-1} \Vert_2$. One readily verifies that the partition $I_j^\star = \{ j \}$, $j = 1, \ldots, n - 2$, and $I_{n-1}^\star = \{ n-1, n \}$ optimizes both the discrete scenario reduction problem in Theorem~\ref{thm:partition-red} as well as the continuous scenario reduction problem in Theorem~\ref{thm:partition}. We thus have $C_2(\hat{\mathbb{P}}_n, n-1) = \frac{1}{\sqrt{2n}}d_{\text{min}}$ and $D_2(\hat{\mathbb{P}}_n, n-1) = \frac{1}{\sqrt{n}}d_{\text{min}}$, which concludes the proof.
	\qed 
\end{proof}


Hence, for any empirical distribution $\hat{\mathbb{P}}_n \in \mathcal{P}_\mathrm{E} (\mathbb{R}^d, n)$ the type-$2$ Wasserstein distance between the minimizer of the \emph{discrete} scenario reduction problem and $\hat{\mathbb{P}}_n$ exceeds the Wasserstein distance between the minimizer of the \emph{continuous} scenario reduction problem and $\hat{\mathbb{P}}_n$ by up to $41.4\%$, and the bound is attainable for any~$n,m$.


\subsection{Guarantees for the Type-1 Wasserstein Distance}
\label{section:compare-1}

In analogy to Section~\ref{section:compare-2}, we first bound $\overline{\kappa}_1 (n, m)$ from above (Theorem~\ref{thm:red1-2}) and below (Proposition~\ref{prop:kap1:below}). In contrast to the previous section, we consider an arbitrary norm $\Vert \cdot \Vert$, and we adapt the definition of the geometric median accordingly.


\begin{theorem}
\label{thm:red1-2}
The upper bound $\overline{\kappa}_1 (n,m)$ in~\eqref{eq:ratio} satisfies $\overline{\kappa}_1 (n,m) \leq 2$ whenever $m \in \{ 2, \ldots, n - 2\}$ as well as $\overline{\kappa}_1 (n,1) \leq 2 \left( 1 - \frac{1}{n} \right)$ and $\overline{\kappa}_1 (n, n-1) \leq 1$.
\end{theorem}

\begin{proof} 
We first prove the statement for $m = 1$ (Step~1) and then extend the result to $m \in \{ 2, \ldots, n - 2 \}$ (Step~2) and $m = n - 1$ (Step~3).

\paragraph{Step 1:}
Fix any $\hat{\mathbb{P}}_n \in \mathcal{P}_\mathrm{E}(\mathbb{R}^d, n)$. As in the proof of Theorem~\ref{thm:red2-sqrt(2)}, we can assume that $\text{gmed} (I) = \bm{0}$ and $\frac{1}{n} \sum_{i \in I} \Vert \bm\xi_i \Vert = 1$ by re-positioning and scaling the atoms $\bm\xi_i$ appropriately. Theorem~\ref{thm:partition} and Remark~\ref{rem:center} then imply that for $m = 1$, we have
	\begin{equation*}
	\begin{aligned}
		C_1(\hat{\mathbb{P}}_n,1) \ =\ 
		\frac{1}{n} \sum_{i \in I} \Vert \bm\xi_i - \text{gmed}(I) \Vert = 1.
	\end{aligned}
	\end{equation*}	
Step 1 is thus complete if we can show that $D_1(\hat{\mathbb{P}}_n,1) \leq 2 \left(1 - \frac{1}{n} \right)$. Indeed, we have
	\begin{equation*}
	\begin{aligned}
		D_1(\hat{\mathbb{P}}_n, 1)
		&=\ \min_{j \in I}\ \frac{1}{n} \sum_{i \in I} \Vert \bm\xi_i - \bm\xi_j \Vert \
		=\ \min_{j \in I}\ \frac{1}{n} \sum_{i \in I \setminus \{j\}} \Vert \bm\xi_i - \bm\xi_j \Vert \\
		&\leq\ \min_{j \in I}\ \frac{1}{n} \sum_{i \in I \setminus \{j\}} \left( \Vert \bm\xi_i \Vert + \Vert \bm\xi_j \Vert \right) \
		=\ \min_{j \in I}\ \frac{1}{n} \left( (n-2)\Vert \bm\xi_j \Vert + \sum_{i \in I} \Vert \bm\xi_i \Vert \right) \\
		&=\ \min_{j \in I}\ \frac{1}{n} \left( (n-2)\Vert \bm\xi_j \Vert + n \right) \
		=\ 1 + \frac{n-2}{n} \cdot \min_{j \in I}\ \Vert \bm\xi_j \Vert \
		\leq\ 2 \left( 1 - \frac{1}{n}\right),
	\end{aligned}
	\end{equation*}
	where the two inequalities follow from the triangle inequality and the fact that $\min_{j \in I} \Vert \bm\xi_j \Vert \leq \frac{1}{n} \sum_{i \in I} \Vert \bm\xi_i \Vert = 1$, respectively.

\paragraph{Step 2:}
Fix any $\hat{\mathbb{P}}_n \in \mathcal{P}_\mathrm{E}(\mathbb{R}^d, n)$. Theorem~\ref{thm:partition} and Remark~\ref{rem:center} then imply that
    \begin{equation*}
	\begin{aligned}
		C_1(\hat{\mathbb{P}}_n,m) =
		\min_{\{ I_j\} \in \mathfrak{P}(I,m)}\frac{1}{n} \sum_{j \in J} \sum_{i \in I_j} \Vert \bm\xi_i - \text{gmed}(I_j) \Vert.
	\end{aligned}
	\end{equation*}
    Let $\{ I_j^\star \}$ be an optimal partition for this problem. The same arguments as in the proof of Theorem~\ref{thm:red2-sqrt(2)} show that
    \begin{equation*}
    \begin{aligned}
		D_1(\hat{\mathbb{P}}_n,m) \; &\leq \;
        \sum_{j \in J} \frac{\vert I^\star_j \vert}{n} D_{1,j} \mspace{40mu} \text{with}~~
		D_{1,j} = \min_{j \in I^\star_j} \frac{1}{\vert I^\star_j \vert} \sum_{i \in I^\star_j} \Vert \bm\xi_i - \bm\xi_j \Vert \\
 	    &\leq \; \sum_{j \in J} \frac{\vert I^\star_j \vert}{n} (2 \, C_{1,j}) \quad \text{with}~~
		C_{1,j} = \frac{1}{\vert I^\star_j \vert} \sum_{i \in I^\star_j} \Vert \bm\xi_i - \text{gmed}(I^\star_j) \Vert,
    \end{aligned}
    \end{equation*}	 
    and the last expression is equal to $2 \, C_1(\hat{\mathbb{P}}_n,m)$ by definition of $C_{1,j}$.

\paragraph{Step 3:}
For $n=2$ and $m=n-1=1$, Step~1 shows that $\overline{\kappa}_1 (2,1) \leq 2 \left( 1 - \frac{1}{2} \right) = 1$. For $n > 2$ and $m=n-1$, the statement can be derived in the same way as the third step in the proof of Proposition~\ref{prop:red2-1}. We omit the details for the sake of brevity.
\qed 
\end{proof}

\begin{proposition}
\label{prop:kap1:below}
There is $\hat{\mathbb{P}}_n \in \mathcal{P}_\mathrm{E}(\mathbb{R}^d, n)$ such that $D_1 (\hat{\mathbb{P}}_n, m) = 2\left( 1 - \frac{m}{n}\right) C_1 (\hat{\mathbb{P}}_n, m)$ under the $1$-norm for all $n$ divisible by $2m$, all $m$ and all $d \geq \frac{n}{2m}$.
\end{proposition}

\begin{proof} 
We first prove the statement for $m = 1$ (Step~1) and then extend the result to $m > 1$ (Step~2). Throughout the proof, we set $k = \frac{n}{2m}$ and consider w.l.o.g. the case where $d = k$.

\paragraph{Step 1:}
Fix $\hat{\mathbb{P}}_n \in \mathcal{P}_\mathrm{E} (\mathbb{R}^d, n)$ with the atoms $\bm{\xi}_i = +\mathbf{e}_i$ as well as $\bm{\xi}_{k+i} = -\mathbf{e}_i$, $i = 1, \ldots, k$. The symmetric placement of the atoms implies that $\text{gmed} (I) = \bm{0}$ and hence $C_1(\hat{\mathbb{P}}_n,1) = 1$. Furthermore, we note that $\Vert \bm\xi_i - \bm\xi_j \Vert_1 = 2$ for all $i \neq j$, that is, any two atoms are equidistant from another (see right panel of Figure~\ref{fig:wsss_2_and_1} for an illustration in $\mathbb{R}^3$). By Theorem~\ref{thm:partition-red}, any $1$-point discrete scenario reduction results in a Wasserstein distance of $2 \, \frac{n-1}{n}$ to $\hat{\mathbb{P}}_n$.

\paragraph{Step 2:}
To prove the statement for $m > 1$, we construct an empirical distribution $\hat{\mathbb{P}}_n \in \mathcal{P}_\mathrm{E} (\mathbb{R}^d, n)$ whose atoms satisfy $\text{supp} (\hat{\mathbb{P}}_n) = \bigcup_{j=1}^m (\mathrm{\Xi}_j^+ \cup \mathrm{\Xi}_j^-)$ with $| \mathrm{\Xi}_j^+ | = | \mathrm{\Xi}_j^- | = k$, $j = 1, \ldots, m$. The atoms $\bm{\xi}_{2(j-1)k+1}, \ldots, \bm{\xi}_{2(j-1)k+k}$ in $\mathrm{\Xi}_j^+$ satisfy $\bm{\xi}_{2(j-1)k+i} = +\mathbf{e}_i + jM\mathbf{e}_1$, $i = 1, \ldots, k$, whereas the atoms $\bm{\xi}_{2(j-1)k+k+1}, \ldots, \bm{\xi}_{2jk}$ in $\mathrm{\Xi}_j^-$ satisfy $\bm{\xi}_{2(j-1)k+k+i} = -\mathbf{e}_i + jM\mathbf{e}_1$, $i = 1, \ldots, k$, for any number $M$ satisfying $M > 2n+2$. The same arguments as in the proof of Theorem~\ref{thm:red2-sqrt(2)} show that any optimal partition $\{ I_j^\star \}$ to the discrete scenario reduction problem in Theorem~\ref{thm:partition-red} as well as the continuous scenario reduction problem in Theorem~\ref{thm:partition} consists of the sets indexing the atoms in $\mathrm{\Xi}_j^+ \cup \mathrm{\Xi}_j^-$, $j = 1, \ldots, m$. Step~1 shows that the continuous scenario reduction problem accumulates a Wasserstein distance of $1$ over each set, whereas the discrete scenario reduction problem accumulates a Wasserstein distance of $2 \, \frac{2k-1}{2k}$ over each set. The result then follows from the fact that there are $m$ such sets and hence the ratio of the respective overall Wasserstein distances amounts to $m \big( 2\frac{2k-1}{2k} \big) / m = 2 \left( 1-\frac{m}{n} \right)$.
    \qed
\end{proof}

Theorem~\ref{thm:red1-2} and Proposition~\ref{prop:kap1:below} imply that $\overline{\kappa}_1 (n, m) \in [2 (1 - m/n), 2]$ for all $n$ and $m \in \{ 2, \ldots, n - 2 \}$. For the small ratios $m:n$ commonly used in practice, we thus conclude that the bound is \emph{essentially independent} of both the number of atoms $n$ in the empirical distribution and the number of atoms $m$ in the reduced distribution. We close with an analysis of the lower bound $\underline{\kappa}_1 (n, m)$.

\begin{proposition}
\label{prop:red1-1}
The lower bound $\underline{\kappa}_1 (n, m)$ in~\eqref{eq:ratio} satisfies $\underline{\kappa}_1 (n, m) = 1$ for all $n,m$.
\end{proposition}

\begin{proof} 
The proof widely parallels that of Proposition~\ref{prop:red2-1}, with the difference that the atoms $\bm{\xi}_1, \ldots, \bm{\xi}_n$ of the empirical distribution $\hat{\mathbb{P}}_n$ are placed such that a geometric median (as opposed to the mean) of each subset in the optimal partition coincides with one of the atoms in that subset. This allows both continuous and discrete scenario reduction to choose the same support points for the reduced distribution, hence incurring the same Wasserstein distance. Details are omitted for brevity.
\qed
\end{proof}

In conclusion, for any $\hat{\mathbb{P}}_n \in \mathcal{P}_\mathrm{E} (\mathbb{R}^d, n)$ the type-$1$ Wasserstein distance between the minimizer of the \emph{discrete} scenario reduction problem and $\hat{\mathbb{P}}_n$ exceeds the Wasserstein distance between the minimizer of the \emph{continuous} scenario reduction problem and $\hat{\mathbb{P}}_n$ by up to $100\%$, and this bound is asymptotically attained for decreasing ratios $m : n$.


\section{Solution Methods}
\label{section:algorithm}

We now review existing and propose new solution schemes for the discrete and continuous scenario reduction problems. More precisely, we will study two heuristics for discrete and continuous scenario reduction, respectively, that do not come with approximation guarantees (Section~\ref{sec:pure_heuristics}), we will propose a constant-factor approximation scheme for both the discrete and the continuous scenario reduction problem (Section~\ref{sec:constant_fac}), and we will discuss two exact reformulations of these problems as mixed-integer optimization problems (Section~\ref{sec:mixed_integer}).

In the remainder of this section, we denote by $D_l (\hat{\mathbb{P}}_n, \mathrm{\Xi})$ the type-$l$ Wasserstein distance between $\hat{\mathbb{P}}_n$ and its closest distribution supported on the finite set~$\mathrm{\Xi}$. Moreover, for an algorithm providing an upper bound $\overline{D}_l (\hat{\mathbb{P}}_n, m)$ on the discrete scenario reduction problem in $\mathbb{R}^d$, we define the algorithm's \emph{approximation ratio} as the maximum fraction $\overline{D}_l (\hat{\mathbb{P}}_n, m) / D_l (\hat{\mathbb{P}}_n, m)$, where the maximum is taken over all $n$ and $m$, as well as all empirical distributions $\hat{\mathbb{P}}_n \in \mathcal{P}_\mathrm{E} (\mathbb{R}^d, n)$.


\subsection{Heuristics for the Discrete Scenario Reduction Problem}\label{sec:pure_heuristics}

We review in Section~\ref{sec:dupacova-and-coworkers} a popular heuristic for the discrete scenario reduction problem due to \cite{Dupacova2003}. We will show that despite the simplicity and efficiency of the algorithm, there is no finite upper bound on the algorithm's approximation ratio. In Section~\ref{sec:lloyd} we adapt a widely used clustering heuristic to the continuous scenario reduction problem, and we show that this algorithm's approximation ratio cannot be bounded from above either.

\subsubsection{Dupa{\v{c}}ov{\'a} et al.'s Algorithm}\label{sec:dupacova-and-coworkers}

We outline Dupa{\v{c}}ov{\'a} et al.'s algorithm for the problem $D_l (\hat{\mathbb{P}}_n, m)$ below.

\begin{center}
\fbox{\begin{minipage}{26.5em}
\centering \underline{\textsc{Dupa{\v{c}}ov{\'a} et al.'s algorithm for $D_l (\hat{\mathbb{P}}_n, m)$:}} \\[-2mm]
\begin{enumerate}
	\item
	Initialize the set of atoms in the reduced set as $R \leftarrow \emptyset$. \\[-2mm]
	\item
	Select the next atom to be added to the reduced set as
	\begin{equation*}
	\bm{\zeta} \in \underset{\bm{\zeta} \in \text{supp}(\hat{\mathbb P}_n )}{\argmin} \ D_l (\hat{\mathbb{P}}_n, R \cup \{ \bm{\zeta} \})
	\end{equation*}
	and update $R \leftarrow R \cup \{ \bm{\zeta} \}$. \\
	\item
	Repeat Step 2 until $\vert R \vert = m$.
\end{enumerate} 
\end{minipage}}
\end{center}

Given an empirical distribution $\hat{\mathbb{P}}_n \in \mathcal{P}_\mathrm{E} (\mathbb{R}^d, n)$, the algorithm iteratively populates the reduced set $R$ containing the atoms of the reduced distribution~$\mathbb{Q}$. Each atom $\bm{\zeta} \in \text{supp}(\hat{\mathbb P}_n)$ is selected greedily so as to minimize the Wasserstein distance between $\hat{\mathbb{P}}_n$ and the closest distribution supported on the augmented reduced set $R \cup \{ \bm{\zeta} \}$. After termination, the distribution $\mathbb{Q}$ can be recovered from the reduced set $R$ as follows. Let $\{ I_{\bm{\zeta}} \} \in \mathfrak{P} (I, m)$ be any partition of $\text{supp}(\hat{\mathbb P}_n)$ into sets $I_{\bm{\zeta}}$, $\bm{\zeta} \in R$, such that $I_{\bm{\zeta}}$ contains all elements of $\text{supp}(\hat{\mathbb P}_n)$ that are closest to $\bm{\zeta}$ (ties may be broken arbitrarily). Then $\mathbb{Q} = \sum_{\bm{\zeta} \in R} q_{\bm{\zeta}} \delta_{\bm{\zeta}}$, where $q_{\bm{\zeta}} = | I_{\bm{\zeta}} | / n$.

\begin{theorem}
\label{thm:wc-dupa}
For every $d \geq 2$ and $l, p \geq 1$, the approximation ratio Dupa{\v{c}}ov{\'a} et al.'s algorithm is unbounded.
\end{theorem}

\begin{proof} 
The proof constructs a specific distribution $\hat{\mathbb{P}}_n$ (Step~1), bounds $D_l (\hat{\mathbb{P}}_n, m)$ from above (Step~2) and bounds the Wasserstein distance between $\hat{\mathbb{P}}_n$ and the output $\mathbb{Q}$ of Dupa{\v{c}}ov{\'a} et al.'s algorithm from below (Step~3).

\paragraph{Step 1:}
Fix $d \geq 2$, $l, p \geq 1$ and $m = 4$, and consider the empirical distribution $\hat{\mathbb{P}}_n \in \mathcal{P}_\mathrm{E} (\mathbb{R}^d, n)$ with $n = 4z+1$ for some positive integer $z$ as well as $\text{supp} (\hat{\mathbb{P}}_n) = \mathrm{\Xi}_1 \cup \cdots \cup \mathrm{\Xi}_4 \cup \{ \bm{\xi}_{4z+1} \}$ with $\mathrm{\Xi}_j = \{ \bm{\xi}_{(j-1)z+1}, \ldots, \bm{\xi}_{jz} \}$, $j = 1, \ldots, 4$, and
\begin{equation*}
\mathrm{\Xi}_1 \subset \mathcal{B}_\epsilon (+\mathbf{e}_1), \quad
\mathrm{\Xi}_2 \subset \mathcal{B}_\epsilon (-\mathbf{e}_1), \quad
\mathrm{\Xi}_3 \subset \mathcal{B}_\epsilon (+\mathbf{e}_2), \quad
\mathrm{\Xi}_4 \subset \mathcal{B}_\epsilon (-\mathbf{e}_2)
\end{equation*}
and $\bm{\xi}_{4z+1} = \bm{0}$, where $\mathcal{B}_\epsilon (\bm{x}) = \{ \bm{\xi} \in \mathbb{R}^d \, : \, \Vert \bm{\xi} - \bm{x} \Vert_p \leq \epsilon \}$ denotes the $\epsilon$-ball around $\bm{x}$. Here, $\epsilon > 0$ is small enough so that each atom in $\mathrm{\Xi}_i$ is closer to $\bm{0}$ than to any atom in any of the other sets $\mathrm{\Xi}_j$. The triangle inequality then implies that
\begin{equation}\label{eq:triangle_dupacova}
\begin{array}{l@{\qquad}l}
\displaystyle \Vert \bm{\xi}_i \Vert_p \in [1 - \epsilon, 1 + \epsilon] & \displaystyle \forall \bm{\xi}_i \in \mathrm{\Xi}_1, \\[1mm]
\displaystyle \Vert \bm{\xi}_i - \bm{\xi}_1 \Vert_p \geq 2 - 2 \epsilon & \displaystyle \forall \bm{\xi}_i \in \mathrm{\Xi}_2, \\[1mm]
\displaystyle \Vert \bm{\xi}_i - \bm{\xi}_1 \Vert_p \geq 1 - \epsilon & \displaystyle \forall \bm{\xi}_i \in \mathrm{\Xi}_3 \cup \mathrm{\Xi}_4.
\end{array}
\end{equation}

\paragraph{Step 2:}
By construction, we have that
\begin{align*}
D_l (\hat{\mathbb{P}}_n, 4)
\; & \leq \;
d_l \left( \hat{\mathbb{P}}_n, \; \frac{z+1}{4z+1} \delta_{\bm{\xi}_z} + \frac{z}{4z+1} \delta_{\bm{\xi}_{2z}} + \frac{z}{4z+1} \delta_{\bm{\xi}_{3z}} + \frac{z}{4z+1} \delta_{\bm{\xi}_{4z}} \right) \\
&\leq \;
         \left[ \frac{1}{4z+1} \left( \left[ \sum_{j=1}^4 \sum_{\bm{\xi}_i \in \mathrm{\Xi}_j}
        	\Vert \bm\xi_i - \bm\xi_{jz} \Vert_p^l \right]
          + \Vert \bm\xi_{4z+1} - \bm\xi_z \Vert_p^l \right)\right]^{1/l} \\[2mm]
        &\leq \;
         \left[ \frac{1}{4z+1} \left( \left[ \sum_{j=1}^4 \sum_{\bm{\xi}_i \in \mathrm{\Xi}_j} (2\epsilon)^l \right]
          + (1 + \epsilon)^l \right) \right]^{1/l} \!\!\!\! = \
         \left[ \frac{4z 2^l \epsilon^l + (1+\epsilon)^l}{4z+1} \right]^{1/l},
\end{align*}
where the first inequality holds because $\bm{\xi}_z, \bm{\xi}_{2z}, \bm{\xi}_{3z}, \bm{\xi}_{4z} \in \text{supp}(\hat{\mathbb{P}}_n)$, the second inequality holds since moving the atoms in $\mathrm{\Xi}_j$ to $\bm{\xi}_{jz}$, $j = 1, \ldots, 4$, and $\bm{\xi}_{4z+1}$ to $\bm{\xi}_z$ represents a feasible transportation plan, and the third inequality is due to~\eqref{eq:triangle_dupacova} and the triangle inequality.
    
\paragraph{Step 3:}
We first show that for a sufficiently small $\epsilon > 0$, Dupa{\v{c}}ov{\'a} et al.'s algorithm adds $\bm\xi_{4z+1} = \bm{0}$ to the reduced set $R$ in the first iteration. We then show that under this selection, the output $\mathbb{Q}$ of Dupa{\v{c}}ov{\'a} et al.'s algorithm can be arbitrarily worse than the bound on $D_l (\hat{\mathbb{P}}_n, 4)$ determined in the previous step.

To show the first point, the symmetry inherent in $\text{supp}(\hat{\mathbb{P}}_n)$ implies that it suffices to show that $d_l (\hat{\mathbb{P}}_n, \delta_{\bm{0}}) < d_l(\hat{\mathbb{P}}_n, \delta_{\bm{\xi}_{1}})$. To this end, we note that
    \begin{equation*}
    \begin{aligned}
    	d^l_l (\hat{\mathbb{P}}_n, \delta_{\bm{0}}) \; = \;
        	\frac{1}{4z+1} \sum_{j=1}^4 \sum_{\bm{\xi}_i \in \mathrm{\Xi}_j} \Vert \bm\xi_i \Vert_p^l \; \leq \;
            \frac{1}{4z+1} \sum_{j=1}^4 \sum_{\bm{\xi}_i \in \mathrm{\Xi}_j} (1 + \epsilon)^l \; = \; 
            \frac{4z}{4z+1}(1 + \epsilon)^l
    \end{aligned}
    \end{equation*}
due to equation~\eqref{eq:triangle_dupacova}, while at the same time
	\begin{equation*}
    \begin{aligned}
    	d^l_l (\hat{\mathbb{P}}_n, \delta_{\bm\xi_1}) \; &= \;
        	\frac{1}{4z+1} \sum_{i=2}^{4z+1} \Vert \bm\xi_i - \bm\xi_1 \Vert_p^l \\
            &\geq \;
            \frac{1}{4z+1} \sum_{i=z+1}^{4z+1} \Vert \bm\xi_i - \bm\xi_1 \Vert_p^l \; \geq \;
            \frac{z(2-2\epsilon)^l + (2z+1)(1 - \epsilon)^l}{4z+1}.
    \end{aligned}
    \end{equation*}
    As $\epsilon$ tends to $0$, we have that
    \begin{equation*}
    	\lim_{\epsilon \rightarrow 0}\, d_l (\hat{\mathbb{P}}_n, \delta_{\bm{0}}) \; \leq \;
        \left[ \frac{4z}{4z+1} \right]^{1/l} \; < \; 
        \left[ \frac{z2^l  + 2z + 1}{4z+1} \right]^{1/l} \; \leq \;
        \lim_{\epsilon \rightarrow 0}\, d_l (\hat{\mathbb{P}}_n, \delta_{{\bm\xi}_1}),
    \end{equation*}
    where the strict inequality is due to $l \geq 1$. As a consequence, we may conclude that there indeed exists an $\epsilon > 0$ such that Dupa{\v{c}}ov{\'a} et al.'s algorithm adds $\bm\xi_{4z+1} = \bm{0}$ to the reduced set $R$ in the first iteration.
    
As for the second point, we note that after adding $\bm\xi_{4z+1} = \bm{0}$ to the reduced set $R$, there must be at least one subset $\mathrm{\Xi}_j$, $j \in \{ 1, \ldots, 4 \}$, such that no $\bm{\xi}_i \in \mathrm{\Xi}_j$ is contained in the final reduced set $R$. Assume w.l.o.g.~that this is the case for $j = 1$. We then have
    \begin{equation*}
    	d_l(\hat{\mathbb{P}}_n, \mathbb{Q}) \; \geq \; \left[ \frac{1}{4z+1} \sum_{\bm{\xi}_i \in \mathrm{\Xi}_1} \Vert \bm\xi_j - \bm{0} \Vert_p \right]^{1/l} \; \geq \quad \left[ \frac{z(1 - \epsilon)}{4z+1} \right]^{1/l},
    \end{equation*}
and combining this with the result of Step~2, we can conclude that the approximation ratio $d_l(\hat{\mathbb{P}}_n, \mathbb{Q})/D_l(\hat{\mathbb{P}}_n,4)$ approaches $\infty$ as $z \to \infty$ and $z \epsilon^l \to 0$.
\qed
\end{proof}

We remark that the algorithm of Dupa{\v{c}}ov{\'a} et al.~can be improved by adding multiple atoms to the reduced set $R$ in Step~2. Nevertheless, a similar argument as in the proof of Theorem~\ref{thm:wc-dupa} shows that the resulting improved algorithm does not allow for a finite upper bound on the approximation ratio either.

\subsubsection{$k$-Means Clustering Algorithm}\label{sec:lloyd}

The $k$-means clustering algorithm has first been proposed in 1957 for a pulse-code modulation problem~\citep{Lloyd1982}, and it has since then become a widely used heuristic for various classes of clustering problems. It aims to partition a set of observations $\bm{x}_1, \ldots, \bm{x}_n$ into $m$ clusters $S_1, \ldots, S_m$ such that the intra-cluster sums of squared distances are minimized. By generalizing the algorithm to arbitrary powers and norms, we can adapt the algorithm to our continuous scenario reduction problem as follows.


\begin{center}
\fbox{\begin{minipage}{33.0em}
\centering \underline{\textsc{$k$-means clustering algorithm for  $C_l (\hat{\mathbb{P}}_n, m)$:}} \\[-2mm]
\begin{enumerate}
	\item
	Initialize the reduced set $R = \{ \bm\zeta_1, \hdots, \bm\zeta_m \} \subseteq \text{supp} (\hat{\mathbb{P}}_n)$ arbitrarily. \\[-2mm]
	\item
Let $\{ I_j \} \in \mathfrak{P} (I, m)$ be any partition whose sets $I_j$, $j \in J$, contain all atoms of $\text{supp}(\hat{\mathbb P}_n)$ that are closest to $\bm{\zeta}_j$ (ties may be broken arbitrarily).
	\item
	For each $j \in J$, update $\bm\zeta_j \leftarrow \argmin \{ \sum_{i \in I_j} \Vert \bm\xi_i - \bm\zeta \Vert^l \, : \, \bm{\zeta} \in \mathbb{R}^d \}$. \\[-2mm]
	\item
	Repeat Steps 2 and 3 until the reduced set $R$ no longer changes.
\end{enumerate} 
\end{minipage}}
\end{center}

For the empirical distribution $\hat{\mathbb{P}}_n \in \mathcal{P}_\mathrm{E} (\mathbb{R}^d, n)$, the algorithm iteratively updates the reduced set $R$ containing the atoms of the reduced distribution~$\mathbb{Q}$ through a sequence of assignment (Step~2) and update (Step~3) steps. Step~2 assigns each atom $\bm{\xi}_i \in \text{supp}(\hat{\mathbb P}_n)$ of the empirical distribution to the closest atom in the reduced set, and Step~3 updates each atom in the reduced set so as to minimize the sum of $l$-th powers of the distances to its assigned atoms from $\text{supp}(\hat{\mathbb P}_n)$. After termination, the continuously reduced distribution $\mathbb{Q}$ can be recovered from the reduced set $R$ in the same way as in the previous subsection.

Remark~\ref{rem:center} implies that for $l = 2$ and $\Vert \cdot \Vert = \Vert \cdot \Vert_2$, Step~3 reduces to $\bm\zeta_j \leftarrow \frac{1}{\vert I_j \vert}\sum_{i \in I_j} \bm\xi_i$, in which case we recover the classical $k$-means clustering algorithm. Although the algorithm terminates at a local minimum, \cite{Dasgupta2008} has shown that for $l = 2$ and $\Vert \cdot \Vert = \Vert \cdot \Vert_2$, the solution determined by the algorithm can be arbitrarily suboptimal. We now generalize this finding to generic type-$l$ Wasserstein distances induced by arbitrary $p$-norms.


\begin{theorem}
\label{thm:wc-lloyd}
If initialized randomly in Step~1, the approximation ratio of the $k$-means clustering algorithm is unbounded for every $d, l, p \geq 1$ with significant probability.
\end{theorem}

\begin{proof}
In analogy to the proof of Theorem~\ref{thm:wc-dupa}, we construct a specific distribution $\hat{\mathbb{P}}_n$ (Step~1), bound $D_l (\hat{\mathbb{P}}_n, m)$ from above (Step~2) and bound the Wasserstein distance between $\hat{\mathbb{P}}_n$ and the output $\mathbb{Q}$ of the $k$-means algorithm from below (Step~3).
    
\paragraph{Step 1:}
Fix $d, l, p \geq 1$ and m = 3, and consider the empirical distribution $\hat{\mathbb{P}}_n \in \mathcal{P}_\mathrm{E} (\mathbb{R}^d, m)$ with $n = 3 z + 1$ for some positive integer $z$ as well as $\text{supp} (\hat{\mathbb{P}}_n) = \mathrm{\Xi}_1 \cup \mathrm{\Xi}_2 \cup \{ \bm{\xi}_{3z+1} \}$ with $\mathrm{\Xi}_1 = \{ \bm{\xi}_1, \ldots, \bm{\xi}_{2z} \}$, $\mathrm{\Xi}_2 = \{ \bm{\xi}_{2z+1}, \ldots, \bm{\xi}_{3z} \}$ and
\begin{equation}\label{eq:lloyd_example}
\mathrm{\Xi}_1 \subset \mathcal{B}_\epsilon (-\mathbf{e}_1), \quad
\mathrm{\Xi}_2 \subset \mathcal{B}_\epsilon (\bm{0}) \quad \text{for} \quad \epsilon \in (0, 1/4),
\end{equation}
as well as $\bm{\xi}_{3z+1} = \mathbf{e}_1$, where again $\mathcal{B}_\epsilon (\bm{x}) = \{ \bm{\xi} \in \mathbb{R}^d \, : \, \Vert \bm{\xi} - \bm{x} \Vert_p \leq \epsilon \}$. By construction, the distance between any pair of atoms in $\mathrm{\Xi}_j$ is bounded above by $2 \epsilon < \frac{1}{2}$, $j = 1, 2$, whereas the distance between two atoms from $\mathrm{\Xi}_1$ and $\mathrm{\Xi}_2$ is bounded from below by $1 - 2 \epsilon > \frac{1}{2}$.
    
\paragraph{Step 2:}
As similar argument as in the proof of Theorem~\ref{thm:wc-dupa} shows that
\begin{align*}
C_l (\hat{\mathbb{P}}_n, 3)
\; & \leq \; 
d_l \left( \hat{\mathbb{P}}_n, \frac{2k}{3k+1} \delta_{-\mathbf{e}_1} + \frac{k}{3k+1} \delta_{\bm{0}} + \frac{1}{3k+1} \delta_{\mathbf{e}_1} \right) \\
& \leq \;
\left[ \frac{1}{3z+1} \left( \sum_{\bm{\xi}_i \in \mathrm{\Xi}_1} \Vert \bm\xi_i + \mathbf{e}_1 \Vert_p^l + \sum_{\bm{\xi}_i \in \mathrm{\Xi}_2} \Vert \bm\xi_i \Vert_p^l \right) \right]^{1/l} \; \leq \; \left[ \frac{3z \epsilon^l} {3z+1} \right]^{1/l},
\end{align*}
where the last inequality follows from~\eqref{eq:lloyd_example}.
    
\paragraph{Step 3:}
We first show that with significant probability, the algorithm chooses a reduced set $R$ containing two atoms from $\mathrm{\Xi}_1$ and one atom from $\mathrm{\Xi}_2$ in the first step. We then show that under this initialization, the output $\mathbb{Q}$ of the algorithm can be arbitrarily worse than the bound on $C_l (\hat{\mathbb{P}}_n, 3)$ determined above. 

In view of the first point, we note that the probability of the reduced set $R$ containing two atoms from $\mathrm{\Xi}_1$ and one atom from $\mathrm{\Xi}_2$ after the first step is $\binom{2z}{2}\binom{z}{1}/\binom{3z+1}{3}$ and approaches 44.44\% as $z \to \infty$. In the following, we thus assume w.l.o.g.~that $R = \{ \bm{\zeta}_1, \bm{\zeta}_2, \bm{\zeta}_3 \}$ with $\bm\zeta_1, \bm\zeta_2 \in \mathrm{\Xi}_1$ and $\bm{\zeta}_3 \in \mathrm{\Xi}_2$ after the first step.

As for the second point, we note that Step~2 of the algorithm assigns the atoms $\bm{\xi}_i \in \mathrm{\Xi}_1$ to either $\bm{\zeta}_1$ or $\bm{\zeta}_2$, whereas the atoms $\bm{\xi}_i \in \mathrm{\Xi}_2 \cup \{ \bm{\xi}_{3z+1} \}$ are assigned to $\bm{\zeta}_3$. Hence, the update of the reduced set $R$ in the next iteration satisfies $\bm{\zeta}_1, \bm{\zeta}_2 \in \mathcal{B}_\epsilon (-\mathbf{e}_1)$, whereas $\bm{\zeta}_3$ is chosen with respect to the set $\mathrm{\Xi}_2 \cup \{ \bm{\xi}_{3z+1} \}$. The algorithm then terminates in the third iteration as the reduced set $R$ no longer changes. We thus find that
    \begin{equation*}
    \begin{aligned}
    	d_l(\hat{\mathbb{P}}_n, \mathbb{Q}) 
        \; &\geq \;
        \left[ \frac{1}{3z+1} \sum_{\bm{\xi}_i \not \in \mathrm{\Xi}_1} \Vert \bm\xi_i - \bm\zeta_3 \Vert_p^l \right]^{1/l} \\
        \; &\geq \;
        \left[ \frac{1}{3z+1} \Big( \Vert \bm\xi_{3z} - \bm\zeta_3 \Vert_p^l + \Vert \bm\xi_{3z+1} - \bm\zeta_3 \Vert_p^l \Big) \right]^{1/l}.
    \end{aligned}
    \end{equation*}
    Recall that $\bm\xi_{3z} \in \mathcal{B}_\epsilon (\bm{0})$ and $\bm\xi_{3z+1} = \mathbf{e}_1$, which implies that
    \begin{equation*}
    	\Vert \bm\xi_{3z} - \bm{\zeta}_3 \Vert_p + \Vert \bm\xi_{3z+1} - \bm{\zeta}_3 \Vert_p 
        \; \geq \;
        \Vert \bm\xi_{3z+1} - \bm{\xi}_{3z} \Vert_p 
        \; \geq \; 
        1 - \epsilon,
    \end{equation*}
    and that at least one of the two terms $\Vert \bm\xi_{3z} - \bm{\zeta}_3 \Vert_p$ or $\Vert \bm\xi_{3z+1} - \bm{\zeta}_3 \Vert_p$ is greater than~$\frac{1-\epsilon}{2}$. We thus conclude that
    \begin{equation*}
    	d_l(\hat{\mathbb{P}}_n, \mathbb{Q}) 
        \; \geq \;
        \left[ \frac{(1- \epsilon)^l}{2^l (3z+1)} \right]^{1/l},
    \end{equation*}
which by virtue Step~2 implies that $d_l(\hat{\mathbb{P}}_n,\mathbb{Q})/C_l(\hat{\mathbb{P}}_n,3) \rightarrow \infty$ as $\epsilon \to 0$.
\qed
\end{proof}


\subsection{Constant-Factor Approximation for the Scenario Reduction Problem}\label{sec:constant_fac}

We now propose a simple approximation scheme for the discrete scenario reduction problem under the type-$1$ Wasserstein distance whose approximation ratio is bounded from above by $5$. We also show that this algorithm gives rise to an approximation scheme for the \emph{continuous} scenario reduction problem with an approximation ratio of $10$. To our best knowledge, we describe the first constant-factor approximations for the discrete and continuous scenario reduction problems.

Our algorithm follows from the insight that the discrete scenario reduction problem under the type-$1$ Wasserstein distance is equivalent to the $k$-median clustering problem. The $k$-median clustering problem is a variant of the $k$-means clustering problem described in Section~\ref{sec:lloyd}, where the $l$-th power of the norm terms is dropped (\textit{i.e.}, $l=1$). In the following, we adapt a well-known local search algorithm \citep{Arya2004} to our discrete scenario reduction problem:

\begin{center}
\fbox{\begin{minipage}{34em} 
\centering \underline{\textsc{Local search algorithm for $D_l (\hat{\mathbb{P}}_n, m)$:}} \\[-2mm]
\begin{enumerate}
	\item
	Initialize the reduced set $R \subseteq \text{supp} (\hat{\mathbb{P}}_n)$, $|R| = m$, arbitrarily. \\[-2mm]
	\item
	Select the next exchange to be applied to the reduced set as
	\begin{equation*}
	(\bm\zeta, \bm{\zeta}^\prime) \in \argmin \left\{ D_l (\hat{\mathbb{P}}_n, R \cup \{ \bm{\zeta} \} \setminus \{ \bm{\zeta}^\prime \}) \, : \, (\bm{\zeta}, \bm{\zeta}^\prime) \in \big( \text{supp} (\hat{\mathbb{P}}_n) \setminus R \big) \times R \right\},
	\end{equation*}
	and update $R \leftarrow R \cup \{ \bm{\zeta} \} \setminus \{ \bm{\zeta}^\prime \}$ if $D_l (\hat{\mathbb{P}}_n, R \cup \{ \bm{\zeta} \} \setminus \{ \bm{\zeta}^\prime \}) < D_l (\hat{\mathbb{P}}_n, R)$. \\
	\item
	Repeat Step 2 until no further improvement is possible.
\end{enumerate} 
\end{minipage}}
\end{center}

For an empirical distribution $\hat{\mathbb{P}}_n \in \mathcal{P}_\mathrm{E} (\mathbb{R}^d, n)$, the algorithm constructs a sequence of reduced sets $R$ containing the atoms of the reduced distribution~$\mathbb{Q}$. In each iteration, Step~2 selects the exchange $R \cup \{ \bm{\zeta} \} \setminus \{ \bm{\zeta}^\prime \}$, $\bm{\zeta} \in \text{supp} (\hat{\mathbb{P}}_n)$ and $\bm{\zeta}^\prime \in R$, that maximally reduces the Wasserstein distance $D_l (\hat{\mathbb{P}}_n, R)$. For performance reasons, this `best fit' strategy can also be replaced with a `first fit' strategy which conducts the first exchange $R \cup \{ \bm{\zeta} \} \setminus \{ \bm{\zeta}^\prime \}$ found that leads to a reduction of $D_l (\hat{\mathbb{P}}_n, R)$. After termination, the reduced distribution~$\mathbb{Q}$ can be recovered from the reduced set $R$ in the same way as in Section~\ref{sec:dupacova-and-coworkers}.

It follows from \cite{Arya2004} that the above algorithm (with either `best fit' or `first fit') has an approximation ratio of $5$ for the discrete scenario reduction problem for all $d$. We now show that the algorithm also provides solutions to the \emph{continuous} scenario reduction problem with an approximation ratio of at most $10$.

\begin{corollary}
\label{cor:guarantee}
The problems $D_l (\hat{\mathbb{P}}_n, m)$ and $C_l (\hat{\mathbb{P}}_n, m)$ are related as follows.
\begin{enumerate}
\item Any approximation algorithm for $D_2 (\hat{\mathbb{P}}_n, m)$ under the $2$-norm with approximation ratio $\alpha$ gives rise to an approximation algorithm for $C_2 (\hat{\mathbb{P}}_n, m)$ under the $2$-norm with approximation ratio $\sqrt{2} \alpha$.
\item Any approximation algorithm for $D_1 (\hat{\mathbb{P}}_n, m)$ under any norm with approximation ratio $\alpha$ gives rise to an approximation algorithm for $C_1 (\hat{\mathbb{P}}_n, m)$ under the same norm with approximation ratio $2 \alpha$.
\end{enumerate}
\end{corollary}

\begin{proof} 
The two statements follow directly from Theorems~\ref{thm:red2-sqrt(2)} and~\ref{thm:red1-2}, respectively.
\qed
\end{proof}

As presented, the local search algorithm is not guaranteed to terminate in polynomial time. This can be remedied by a variant of the algorithm that only accepts exchanges $R \cup \{ \bm{\zeta} \} \setminus \{ \bm{\zeta}^\prime \}$ that reduce the Wasserstein distance $D_l (\hat{\mathbb{P}}_n, R)$ by at least $\epsilon / ((n-m)m)$ for some constant $\epsilon > 0$. It follows from \cite{Arya2004} that for any $\epsilon$, this variant terminates in polynomial time and provides a $(5 + \epsilon)$-approximation for the discrete scenario reduction problem. The algorithm can also be extended to accommodate multiple swaps in every iteration, which lowers the approximation ratio to $3 + \epsilon$ at the expense of additional computations.

We remark that there is a wealth of algorithms for the $k$-median problem that can be adapted to the discrete scenario reduction problem. For example, \cite{Charikar2012} present a rounding scheme for the $k$-median problem that gives rise to a polynomial-time algorithm for $D_1 (\hat{\mathbb{P}}_n, R)$ with an approximation ratio of $3.25$. Likewise, \cite{Kanungo2004} propose a local search algorithm for the $k$-median problem that gives rise to a polynomial-time algorithm for $D_2 (\hat{\mathbb{P}}_n, R)$ under the $2$-norm with an approximation ratio of $9 + \epsilon$. In both cases, Corollary~\ref{cor:guarantee} allows us to extend these guarantees to the corresponding versions of the continuous scenario reduction problem.


\subsection{Mixed-Integer Reformulations of the Discrete and Continuous Scenario Reduction Problems}\label{sec:mixed_integer}

We first review a well-known mixed-integer linear programming (MILP) reformulation of the discrete scenario reduction problem $D_l (\hat{\mathbb{P}}_n, m)$:

\begin{theorem}
\label{thm:scered_milp}
The discrete scenario reduction problem can be formulated as the MILP
\begin{equation}
\label{opt:scered_milp}
\raisebox{5mm}{$D^l_l(\hat{\mathbb{P}}_n, m) \; = \;$}
\begin{array}{c@{\quad}l}
\displaystyle \min_{\mathbf{\Pi}, \, \bm{\lambda}} & \displaystyle \frac{1}{n} \left< \mathbf{\Pi}, \mathbf{D} \right> \\[2.5mm]
\displaystyle \text{\emph{s.t.}} & \displaystyle \mathbf{\Pi} \mathbf{e} = \mathbf{e}, \;\; \mathbf{\Pi} \leq \mathbf{e} \bm{\lambda}^\top, \;\; \bm{\lambda}^\top \mathbf{e} = m \\
& \displaystyle \mathbf{\Pi} \in \mathbb{R}^{n \times n}_+, \;\; \bm{\lambda} \in \{ 0, 1\}^n,
\end{array}
\end{equation}
with $\mathbf{D} \in \mathbb{S}^n$ and $d_{ij} = \Vert \bm\xi_i - \bm\xi_j \Vert^l$ encoding the distances among the atoms in~$\text{supp}(\hat{\mathbb{P}}_n)$.
\end{theorem}

\begin{proof} 
	See {\em e.g.} \cite{Heitsch2003}. \qed
\end{proof}

In problem~\eqref{opt:scered_milp}, the decision variable $\pi_{ij}$ determines how much of the probability mass of atom $\bm{\xi}_i$ in  $\hat{\mathbb{P}}_n$ is shifted to the atom $\bm{\zeta}_j$ in the reduced distribution $\mathbb{Q}$, whereas the decision variable $\lambda_j$ determines whether the atom $\bm\xi_j \in \text{supp}(\hat{\mathbb{P}}_n)$ is contained in the support of $\mathbb{Q}$. A solution $(\mathbf{\Pi}^\star, \bm{\lambda}^\star)$ to problem~\eqref{opt:scered_milp} allows us to recover the reduced distribution via $\mathbb{Q} = \frac{1}{n} \sum_{j=1}^n \mathbf{e}^\top \mathbf{\Pi}^\star \mathbf{e}_j \cdot \delta_{\bm{\xi}_j}$. Problem~\eqref{opt:scered_milp} has $n$ binary and $n^2$ continuous variables as well as $n^2+n+1$ constraints.

We now consider the \emph{continuous} scenario reduction problem. Due to its bilinear objective function, which involves products of transportation weights $\pi_{ij}$ and the distances $\Vert \bm\xi_i - \bm\zeta_j \Vert^l$ containing the continuous decision variables $\bm{\zeta}_j$, this problem may not appear to be amenable to a reformulation as a mixed-integer convex optimization problem. We now show that such a reformulation indeed exists.

\begin{theorem}
\label{thm:sceopt_miconvex}
The continuous scenario reduction problem can be formulated as the mixed-integer convex optimization problem
\begin{equation}
\label{opt:sceopt_miconvex}
\raisebox{7.25mm}{$C^l_l(\hat{\mathbb{P}}_n, m) \; = \;$}
\begin{array}{c@{\quad}l}
\displaystyle \min_{\mathbf{\Pi}, \, \bm{c}, \, \{\bm\zeta_j\}} & \displaystyle \frac{1}{n} \mathbf{e}^\top \bm{c} \\[2.5mm]
\displaystyle \text{\emph{s.t.}} & \displaystyle \mathbf{\Pi} \mathbf{e} = \mathbf{e} \\
& \displaystyle \Vert \bm\xi_i - \bm\zeta_j \Vert^l \leq c_{i} + M (1 - \pi_{ij}) \quad \forall i \in I, \, \forall j \in J \\
& \displaystyle \mathbf{\Pi} \in \{ 0, 1 \}^{n \times m}_+, \;\; \bm{c} \in \mathbb{R}^n_+, \;\; \bm\zeta_1,\hdots,\bm\zeta_m \in \mathbb{R}^d,
\end{array}
\end{equation}
where $M = \max_{i,j \in I} \Vert \bm\xi_i - \bm\xi_j \Vert^l$ denotes the diameter of the support of $\hat{\mathbb{P}}_n$.
\end{theorem}

In problem~\eqref{opt:sceopt_miconvex}, the decision variable $\pi_{ij}$ determines whether or not the probability mass of atom $\bm{\xi}_i$ in the empirical distribution $\hat{\mathbb{P}}_n$ is shifted to the atom $\bm{\zeta}_j$ in the reduced distribution $\mathbb{Q}$, whereas the decision variable $c_i$ records the cost of moving the atom $\bm\xi_i$ under the transportation plan $\mathbf{\Pi}$.

\begin{proof} \emph{of Theorem~\ref{thm:sceopt_miconvex}} $\,$
We prove the statement by showing that optimal solutions to problem~\eqref{opt:sceopt_miconvex} correspond to feasible solutions in problem~\eqref{eq:voronoi-1} with the same objective function value in their respective problems and vice versa.

Fix a minimizer $(\mathbf{\Pi}^\star, \bm{c}^\star, \bm\zeta_1^\star,\hdots,\bm\zeta_m^\star)$ to problem~\eqref{opt:sceopt_miconvex}, which corresponds to a feasible solution $(\{ I_j \}, \bm\zeta_1^\star,\hdots,\bm\zeta_m^\star)$ in problem~\eqref{eq:voronoi-1} if we set $I_j = \{ i \in I \, : \, \pi_{ij}^\star = 1 \}$ for all $j \in J$. Note that $c_i^\star = \Vert \bm\xi_i - \bm\zeta_j^\star \Vert^l$ for $j \in J$ and $i \in I_j$. Thus, both solutions adopt the same objective value.

Conversely, fix a minimizer $(\{ I_j^\star \}, \bm\zeta_1^\star,\hdots,\bm\zeta_m^\star)$ to problem~\eqref{eq:voronoi-1}. This solution corresponds to a feasible solution $(\mathbf{\Pi}, \bm{c}, \bm\zeta_1^\star,\hdots,\bm\zeta_m^\star)$ to problem~\eqref{opt:sceopt_miconvex} if we set $\pi_{ij} = 1$ if $i \in I_j^\star$ and $\pi_{ij} = 0$ otherwise for all $j \in J$, as well as $c_i = \Vert \bm\xi_i - \bm\zeta_j^\star \Vert^l$ for all $i \in I_j$ and $j \in J$. By construction, both solutions adopt the same objective value.
\qed
\end{proof}

A solution $(\mathbf{\Pi}^\star, \bm{c}^\star, \bm{\zeta}^\star_1, \ldots, \bm{\zeta}^\star_m)$ to problem~\eqref{opt:sceopt_miconvex} allows us to recover the reduced distribution via $\mathbb{Q} = \frac{1}{n} \sum_{j=1}^m \mathbf{e}^\top \mathbf{\Pi}^\star \mathbf{e}_j \cdot \delta_{\bm{\zeta}^\star_j}$. Problem~\eqref{opt:sceopt_miconvex} has $nm$ binary and $n+md$ continuous variables as well as $nm+n$ constraints. We now show that~\eqref{opt:sceopt_miconvex} typically reduces to an MILP or a mixed-integer second-order cone program (MISOCP).

\begin{proposition}
\label{prop:sceopt_misocp}
For the type-$1$ Wasserstein distance induced by $\Vert \cdot \Vert_1$ or $\Vert \cdot \Vert_\infty$, problem~\eqref{opt:sceopt_miconvex} reduces to an MILP.  
For any type-$l$ Wasserstein distance induced by $\Vert \cdot \Vert_p$, where $l \geq 1$ and $p \geq 1$ are rational numbers, problem~\eqref{opt:sceopt_miconvex} reduces to an MISOCP.
\end{proposition}

\begin{proof} 
In view of the first statement, we note that $\Vert \bm\xi_i - \bm\zeta_j \Vert_1 \leq c_{i} + M (1 - \pi_{ij})$ is satisfied if and only if there is $\bm{\phi}_{ij} \in \mathbb{R}^d$ such that
\begin{equation*}
\bm{\phi}_{ij} \geq \bm\xi_i - \bm\zeta_j, \;\; \bm{\phi}_{ij} \geq \bm\zeta_j - \bm\xi_i \;\; \text{and} \;\; \mathbf{e}^\top \bm{\phi}_{ij} \leq c_{i} + M (1 - \pi_{ij}).
\end{equation*}
Likewise, $\Vert \bm\xi_i - \bm\zeta_j \Vert_\infty \leq c_{i} + M (1 - \pi_{ij})$ holds if and only if there is $\phi_{ij} \in \mathbb{R}$ with
\begin{equation*}
\phi_{ij} \mathbf{e} \geq \bm\xi_i - \bm\zeta_j, \;\; \phi_{ij} \mathbf{e} \geq \bm\zeta_j - \bm\xi_i \;\; \text{and} \;\; \phi_{ij} \leq c_{i} + M (1 - \pi_{ij}).
\end{equation*}

As for the second statement, we note that $\Vert \bm\xi_i - \bm\zeta_j \Vert_p^l \leq c_{i} + M (1 - \pi_{ij})$ is satisfied if and only if there is $\phi_{ij} \in \mathbb{R}$ such that
\begin{equation*}
\phi_{ij} \geq \Vert \bm\xi_i - \bm\zeta_j \Vert_p \;\; \text{and} \;\; \phi_{ij}^l \leq c_{i} + M (1 - \pi_{ij}).
\end{equation*}
For rational $l, p \geq 1$, both inequalities can be expressed through finitely many second-order cone constraints \citep[Section~2.3]{Alizadeh2003}.
\qed
\end{proof}

\section{Numerical Experiment: Color Quantization}
\label{section:experiment}

Color quantization aims to reduce the color palette of a digital image without compromising its visual appearance. In the standard RGB24 model colors are encoded by vectors of the form $(r, g, b) \in \{ 0, 1, \hdots, 255 \}^3$. This means that the RGB24 model can represent a vast number of 16$,$777$,$216 distinct colors. Consequently, color quantization serves primarily as a lossy image compression method.

In the following we interpret the color quantization problem as a discrete scenario reduction problem using the type-1 Wasserstein distance induced by the 1-norm on $\mathbb R^3$. Thus, we can solve color quantization problems via Dupa{\v{c}}ov{\'a}'s greedy heuristic, the local search algorithm or the exact MILP reformulation~\eqref{opt:scered_milp}. In our experiment we aim to compress all 24 pictures from the Kodak Lossless True Color Image Suite (\url{http://r0k.us/graphics/kodak/}) to~$m = 2^1, \hdots, 2^9$ colors. As the MILP reformulation scales poorly with $n$, we first reduce each image to $n \, \lesssim \, $1$,$024 colors using the Linux command ``convert -colors'', which is distributed through ImageMagick (\url{https://www.imagemagick.org}). We henceforth refer to the resulting 1$,$024-color images as the originals. 

 In all experiments we use an efficient variant of Dupa{\v{c}}ov{\'a}'s algorithm due to \cite{Heitsch2003} (\texttt{DPCV}), and we initialize the local search algorithm either with the color palette obtained from Dupa{\v{c}}ov{\'a}'s algorithm (\texttt{LOC-1}) or na\"ively with the $m$ most frequent colors of the original image (\texttt{LOC-2}). The MILP~\eqref{opt:scered_milp} is solved with GUROBI~7.0.1 (\texttt{MILP}). All algorithms are implemented in C++, and all experiments are executed on a 3.40GHz i7 CPU machine with 16GB RAM. We report the average and the worst-case runtimes in Table~\ref{tab:times}. Note that \texttt{DPCV}, \texttt{LOC-1} and \texttt{LOC-2} all terminate in less than 14 seconds across all instances, while \texttt{MILP} requires substantially more time (the maximum runtime was set to ten hours). Moreover, warmstarting the local search algorithm with the color palette obtained from \texttt{DPCV} can significantly reduce the runtimes.

\begin{table}[h]
\centering
\begin{tabular}{r|CCCC}
    & \texttt{DPCV} & \texttt{LOC-1} & \texttt{LOC-2} & \texttt{MILP} \\ \hline
	Average (secs)  & 2.16 & 2.60 & 3.59 & 1$,$349.71   \\
	Worst-case (secs) & 8.21 & 9.89 & 13.69 & 36$,$120.99   \\ \hline \hline
\end{tabular}
\caption{Runtimes of different methods for discrete scenario reduction. \label{tab:times}}
\end{table}

\begin{figure}[h]
 	\centering
    \begin{subfigure}[t]{0.3\textwidth}
    	\includegraphics[width=0.18\paperwidth]{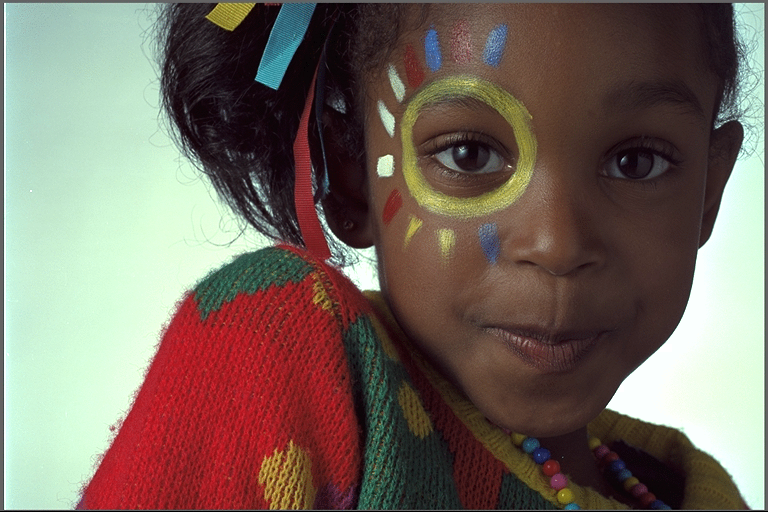}
        \caption{Original}
    \end{subfigure}
    ~~
    \begin{subfigure}[t]{0.3\textwidth}
    	\includegraphics[width=0.18\paperwidth]{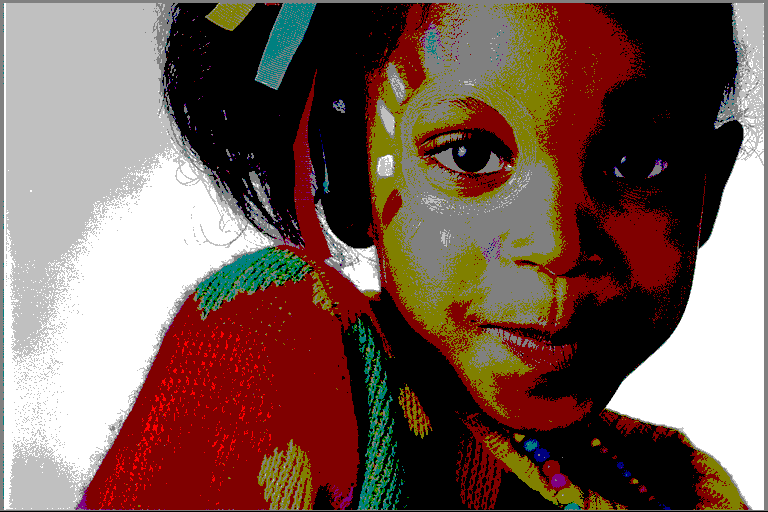}
        \caption{\texttt{MS Paint}}
    \end{subfigure}
    ~~
    \begin{subfigure}[t]{0.3\textwidth}
    	\includegraphics[width=0.18\paperwidth]{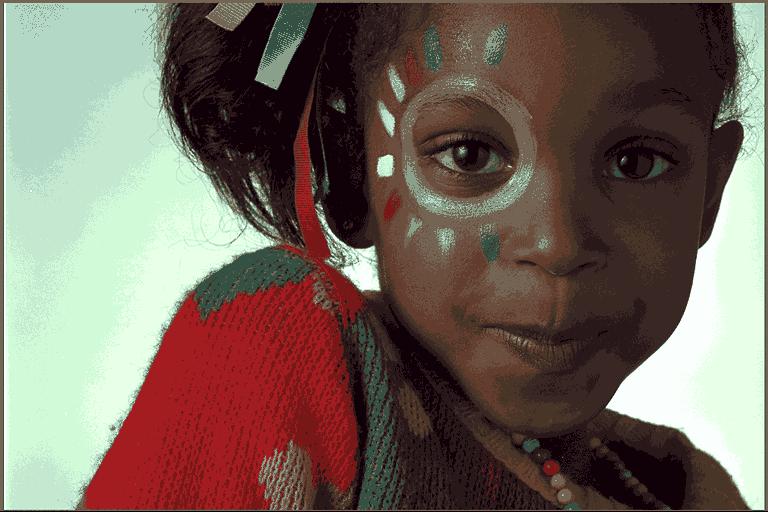}
        \caption{\texttt{DPCV}}
    \end{subfigure}
    ~\\[5mm]
    \begin{subfigure}[t]{0.3\textwidth}
    	\includegraphics[width=0.18\paperwidth]{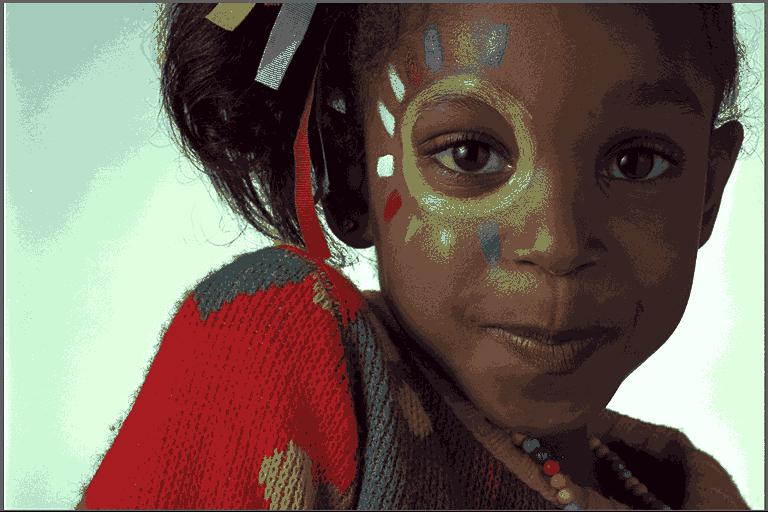}
        \caption{\texttt{LOC-1}}
    \end{subfigure}
    ~~
    \begin{subfigure}[t]{0.3\textwidth}
    	\includegraphics[width=0.18\paperwidth]{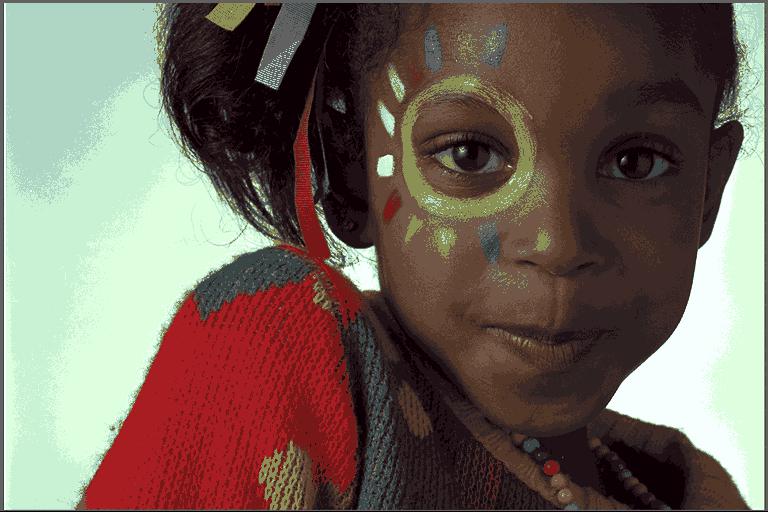}
        \caption{\texttt{LOC-2}}
    \end{subfigure}
    ~~
    \begin{subfigure}[t]{0.3\textwidth}
    	\includegraphics[width=0.18\paperwidth]{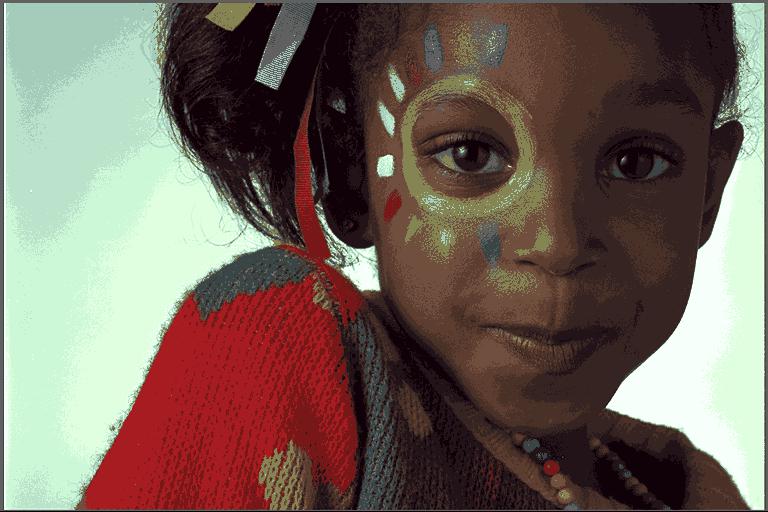}
        \caption{\texttt{MILP}}
    \end{subfigure}
    ~\\[5mm]
    \caption{Outputs of different color quantization algorithms for image ``kodim15.png''.}
    \label{fig:kodim}
\end{figure}

As an example, Figure~\ref{fig:kodim} shows the image~``kodim15.png'' as well as the results of different color quantization algorithms for $m=16$. While the outputs of \texttt{LOC-1}, \texttt{LOC-2} and \texttt{MILP} are almost indistinguishable, the output of \texttt{DPCV} has ostensible deficiencies ({\em e.g.}, it misrepresents the yellow color around the subject's eye). For comparison, we also show the output of the color quantization routine in Microsoft Paint (\texttt{MS Paint}). Figure~\ref{fig:compare} visualizes the optimality gaps of \texttt{DPCV}, \texttt{LOC-1} and \texttt{LOC-2} relative to \texttt{MILP} ({\em i.e.}, their respective approximation ratio $-$ $1$). Our experiment suggests that the local search algorithm is competitive with \texttt{MILP} in terms of output quality but at significantly reduced runtimes. Moreover, the local search algorithm \texttt{LOC-1} warmstarted with the color palette obtained from \texttt{DPCV} is guaranteed to outperform \texttt{DPCV} in terms of  optimality gaps.

\paragraph{Acknowledgements} This research was funded by the SNSF grant BSCGI0\_157733 and the EPSRC grants EP/M028240/1 and EP/M027856/1.

\begin{figure}[h]
    \includegraphics[width=0.57\paperwidth]{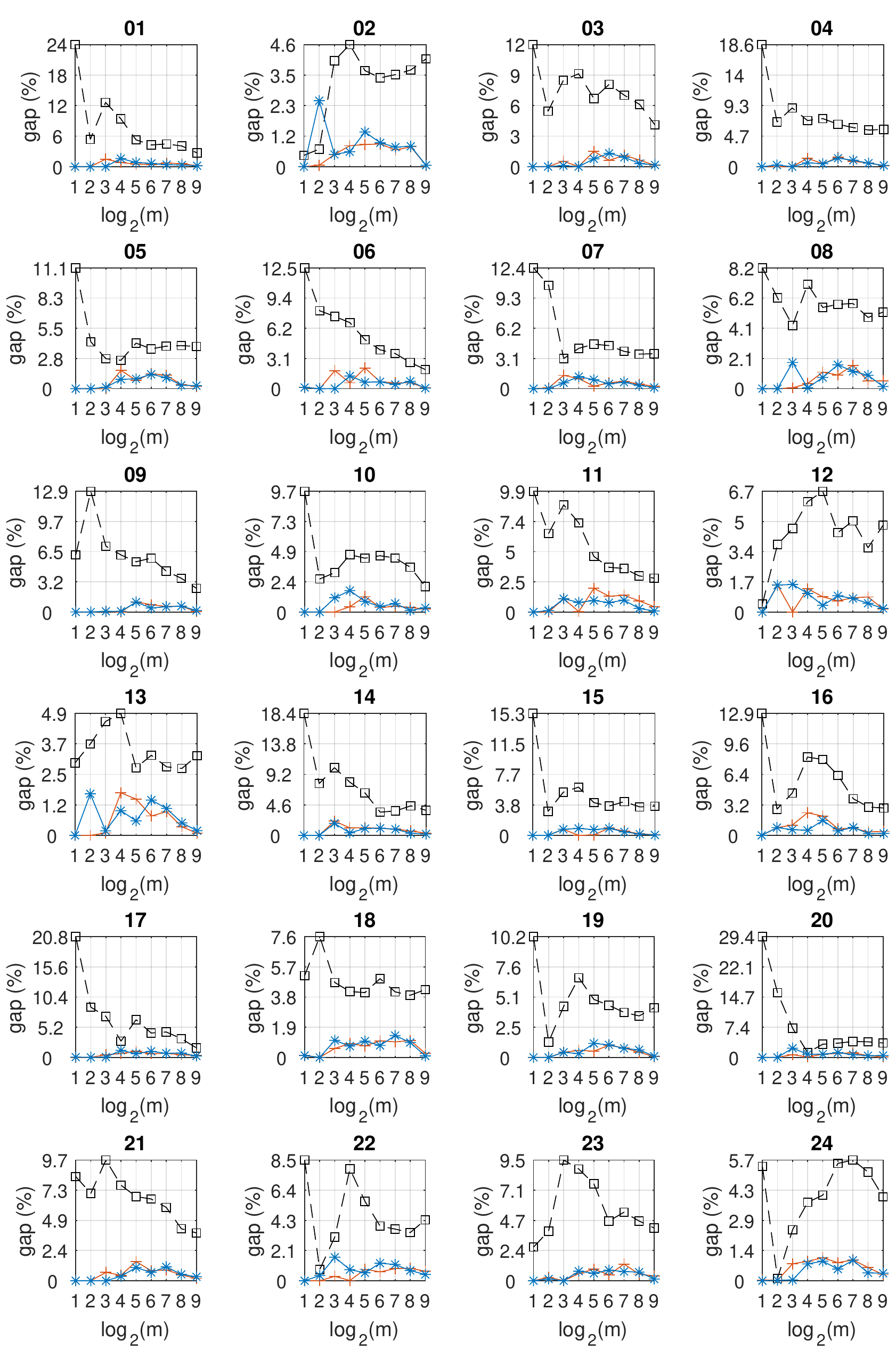}
    \caption{Optimality gaps of \texttt{DPCV} (dashed line with box), \texttt{LOC-1} (solid line with plus) and \texttt{LOC-2} (solid line with star) relative to \texttt{MILP}.}
    \label{fig:compare}
\end{figure}



\bibliographystyle{spbasic}      
\bibliography{references}

\newpage

\section*{Appendix: Auxiliary Results}

The proofs of Theorem~\ref{thm:wcwd-2} relies on the following two lemmas.

\begin{lemma}
\label{lem:symmetry}
The semidefinite program~\eqref{opt:symmetric_sdp} admits an optimal solution $(\tau, \mathbf{S})$ with $\mathbf{S} = \alpha \mathbb{I} + \beta \bm{11}^\top$ for some $\alpha, \beta \in \mathbb{R}$.
\end{lemma}
\begin{proof}
	Let $(\tau, \mathbf{S}^\star)$ be any optimal solution to~\eqref{opt:symmetric_sdp}, which exists because~\eqref{opt:symmetric_sdp} has a continuous objective function and a compact feasible set, and denote by $\mathfrak{S}$ the set of all permutations of~$I$. For any $\sigma\in\mathfrak{S}$, the permuted solution $(\tau, \mathbf{S}^\sigma)$, with $s^\sigma_{ij} = s^\star_{\sigma(i)\sigma(j)}$ is also optimal in~\eqref{opt:symmetric_sdp}. Note first that $(\tau, \mathbf{S}^\sigma)$ is feasible in~\eqref{opt:symmetric_sdp} because
\begin{equation*}
    \begin{aligned}
    	&\tau \leq \sum_{j \in J} \frac{1}{\vert I_j \vert^2} \sum_{i \in I_j} \Bigg( \vert I_j \vert^2 s^\sigma_{ii} - 2\vert I_j \vert \sum_{k \in I_j} s^\sigma_{ik} + \sum_{k \in I_j} s^\sigma_{kk} + \sum_{\substack{k,k^\prime \in I_j \\ k \neq k^\prime}} s^\sigma_{k k^\prime} \Bigg) \\
        \iff \ &\tau \leq \sum_{j \in J} \frac{1}{\vert I^\sigma_j \vert^2} \sum_{i \in I^\sigma_j} \Bigg( \vert I^\sigma_j \vert^2 s^\star_{ii} - 2\vert I^\sigma_j \vert \sum_{k \in I_j} s^\star_{ik} + \sum_{k \in I_j} s^\star_{kk} + \sum_{\substack{k,k^\prime \in I_j \\ k \neq k^\prime}} s^\star_{k k^\prime} \Bigg) ,
    \end{aligned}
    \end{equation*}
    where the index sets $I^\sigma_j=\{\sigma(i):~ i \in I_j \}$ for $j\in J$ form an $m$-set partition from within $\mathfrak{P}(I,m)$, and because $\mathbf{S}^\sigma\succeq \bm{0}$ and $s^\sigma_{ii} = s^\star_{\sigma(i)\sigma(i)}\leq 1$ for all $i\in I$ by construction. Moreover, it is clear that $(\tau, \mathbf{S}^\sigma)$ and $(\tau, \mathbf{S}^\star)$ share the same objective value in~\eqref{opt:symmetric_sdp}. Thus, $(\tau, \mathbf{S}^\sigma)$ is optimal in~\eqref{opt:symmetric_sdp} for every $\sigma\in\mathfrak{S}$.

    The convexity of problem~\eqref{opt:symmetric_sdp} implies that $(\tau,\mathbf{S})$ with $\mathbf{S}=\frac{1}{n!} \sum_{\sigma \in \mathfrak{S}} \mathbf{S}^\sigma$ is also optimal in~\eqref{opt:symmetric_sdp}. The claim follows by noting that $\mathbf{S}$ is invariant under permutations of the coordinates and thus representable as $\alpha\mathbb{I} + \beta\bm{11}^\top$ for some $\alpha, \beta \in \mathbb{R}$. 
    \qed
\end{proof}

\begin{lemma}
\label{lem:eigen}
For $\alpha, \beta \in \mathbb{R}$ the eigenvalues of $\mathbf{S} = \alpha \mathbb{I} + \beta \bm{11}^\top \in \mathbb{S}^n$ are given by $\alpha + n \beta$ (with multiplicity 1) and $\alpha$ (with multiplicity $n-1$).
\end{lemma}
\begin{proof}
	Note that $\mathbf{S}$ is a circulant matrix, meaning that each of its rows coincides with the preceding row rotated by one element to the right. Thus, the eigenvalues of $\mathbf{S}$ are given by $\alpha + \beta(1 + \rho_j^1 + \hdots \rho_j^{n-1})$, $j=0,\ldots,n-1$, where $\rho_j=e^{2\pi i j/n}$ and~$i$ denotes the imaginary unit; see {\em e.g.} \cite{Gray01}. For $j=0$ we then obtain the eigenvalue $\alpha + n\beta$, and for $j= 1,\ldots, n-1$ we obtain the other $n-1$ eigenvalues, all of which equal $\alpha$ because $\sum_{k=0}^{n-1}e^{2\pi i jk/n}=(1-e^{2\pi i j})/(1-e^{2\pi i j/n})=0$. \qed
\end{proof}



\end{document}